\documentclass[11pt]{amsart}
\usepackage[margin=1.2in]{geometry}
\usepackage{amsmath}

\usepackage{amsthm}
\usepackage{hyperref}
\usepackage{soul}
\usepackage{amsfonts}
\usepackage{color}
\usepackage{amssymb}
\usepackage{amscd}
\numberwithin{equation}{section}
\newtheorem{theorem}{Theorem}[section]

\newtheorem{remark}[theorem]{Remark}
\newtheorem{proposition}[theorem]{Proposition}
\newtheorem{lemma}[theorem]{Lemma}
\newtheorem{prop}[theorem]{Proposition}

\theoremstyle{definition}
\newtheorem{defi}{Definition}[section]
\newtheorem{rem}[defi]{Remark}

\newcommand\be{\beta}

\newcommand\g{\mathfrak g}

\newcommand\h{\mathfrak h}

\newcommand\D{\Delta}
\renewcommand\l{\lambda}
\newcommand\Dp{\Delta^+}

\renewcommand\d{\delta}

\renewcommand\a{\alpha}
\renewcommand\aa{\mathfrak a}

\renewcommand\k{\mathfrak k}

\newcommand\ganz{\mathbb Z}

\newcommand\s{\sigma}

\renewcommand\aa{\mathfrak a}
\newcommand\la{\lambda}

\newcommand\C{\mathbb C}

\newcommand\p{\mathfrak p}

\newcommand{\bea}{\begin{eqnarray}}
\newcommand{\eea}{\end{eqnarray}}

\newcommand{\trei}{{x_i}}
\newcommand{\trej}{{x_j}}

\newcommand{\ci}{{u_i}}
\newcommand{\cj}{{u_j}}

\newcommand{\dduej}{{u^j}}

\newcommand{\dci}{{(-[x_\theta,u^i])}}

\begin{document}
\title{Yangians vs minimal $W$-algebras: a surprizing coincidence}
\author[Victor~G. Kac, Pierluigi M\"oseneder Frajria,  Paolo  Papi]{Victor~G. Kac\\Pierluigi M\"oseneder Frajria\\Paolo  Papi}
\begin{abstract} We prove that the  singularities of the $R$-matrix $R(k)$ of the minimal quantization of the adjoint representation of the Yangian $Y(\g)$ of a  finite dimensional simple Lie algebra $\g$ are the opposite of the roots of the monic polynomial  $p(k)$ entering in the OPE expansions of quantum  fields of conformal weight $3/2$ of the universal  minimal affine $W$-algebra at level $k$ attached to  $\g$.
\end{abstract}
\maketitle
\section{Introduction}
Let $\g$ be  a finite dimensional simple Lie algebra over $\mathbb C$, different from $sl(2)$. Let   $Y(\g)$ be  Drinfeld's Yangian associated to $\g$ and $W^k(\g,\theta)$ the universal  minimal affine $W$-algebra at level $k$. The purpose of this paper is to explain a remarkable coincidence  arising when considering, on the one hand, the minimal quantization to $Y(\g)$ of the adjoint representation of $\g$, and,  on the other hand, the OPE expansion for primary fields of  $W^k(\g,\theta)$ of conformal weight $3/2$.
To explain more precisely this coincidence we need some recollections. The algebra $Y(\g)$  is  a Hopf algebra deformation of $U(\g[t])$ which has been introduced in the famous paper \cite{Dr1} to construct solutions of  the quantum Yang-Baxter equation.
Its presentation involves  generators $X$ and $J(X), X\in \g$. In \cite[Theorem 8]{Dr1} Drinfeld explains how to quantize the adjoint representation of $\g$ to $Y(\g)$: the ``minimal'' way  of getting this quantization is to consider the space $V=\g\oplus \C$ and let the generators $X,\,X\in\g$ act in the  natural way, and the generators $J(X),\,X\in\g$ act by  \eqref{actionY}. Formula \eqref{actionY} involves a certain constant $\d$, which depends just on the choice of the bilinear invariant form on $\g$. 
Expanding on Drinfeld's work, Chari and Pressley \cite{ChPr1} studied the R-matrix $R(k)$ associated to $V$, and found that the blocks of this matrix corresponding the trivial and the adjoint isotypic components have as singularities $1$, $h^\vee/2$, and the roots $r_1,r_2$ of a degree two monic polynomial in $k$. Here and further $h^\vee$ is the dual Coxeter number of $\g$. It is implicit in their analysis that $\d=-\tfrac{1}{2} r_1r_2$. See Remark \ref{CCPP}.

Kac, Roan and Wakimoto \cite{KRW}  associated  a vertex algebra $W^{k} (\g, f)$, called a {\it universal affine W-algebra},  to each triple $(\g, f, k)$, where  $f$ is a nilpotent element of $\g$ viewed up to conjugation, and $k\in \C$, by applying the quantum Hamiltonian reduction functor  to the affine vertex algebra $V^k(\g)$.
In particular, it was shown that, for $k\ne -h^\vee$,  ${W}^{k}(\g,f)$ has  a  set of free generators, including a Virasoro vector $\omega$.  A more explicit presentation has  been obtained in \cite{KW} when $f$ is an element from the minimal non-zero nilpotent orbit. Since $e_{-\theta}$, a root vector attached to the minimal root $-\theta$, is such an element, we  will denote this vertex algebra  by $W^k(\g,\theta)$. A further improvement  has been obtained in  \cite{AKMPP}, where it has been proved that the OPE expansion of quantum fields of conformal weight $3/2$  depends on a canonical monic quadratic polynomial $p(k)$. The surprizing fact is that its roots are $-r_1$ and $-r_2$.

Our approach to the explanation is essentially Lie-theoretic, even if information coming from the structure of vertex operator algebras generated by fields of low conformal weight is  needed. Our first step is to provide a proof of Theorem 8 in \cite{Dr1} convenient for our goals. 
This is based on the analysis of certain $\g$-equivariant maps  $G_2:\bigwedge^2\g\to S^3(\g^*)$ and  $G_3:\bigwedge^3\g\to S^3(\g^*)$ (cf. \eqref{G23}), which arise naturally when considering Drinfeld's formula \eqref{actionY}. 
The crucial  Lemma \ref{DeC} has been suggested to us by 
Corrado De Concini.
Along the way  we obtain  a uniform formula for $\d$, see \eqref{finaledelta}, \eqref{altraformula}, which easily specialize to Drinfeld's expressions for $\d$ in each type of $\g$. Remark that the handier formula \eqref{finaledelta} is given in terms of the grading \eqref{gradazione} on $\g$ associated to an element from the  minimal nilpotent  orbit.

On the $W$-algebra side, we consider the  grading \eqref{gradazione}  and investigate the possible    vertex algebras  generated by fields $L$, $J^v$ with  $v\in \g^\natural$ (cf. \eqref{gnatural}), $G^u$ with $u\in \g_{-1/2}$,
with the following  $\lambda$-brackets: $L$ is a Virasoro vector with central charge $\frac{k\,\dim\g}{k+h^\vee}-6k+h^\vee-4$, $J^u$ are primary of conformal weight $1$, $G^{v}$ are primary of conformal weight $\frac{3}{2}$, the $J^u$ generate an affine vertex algebra, and no other constraints. The existence of such vertex algebras is guaranteed by \cite{KW}.
The final outcome is that imposing Jacobi identity, up to an overall  multiplicative constant, one obtains precisely the relations given by \cite{KW}: see Proposition \ref{propfinale}. Coming back to the explanation of the coincidence, one substitutes  auxiliary relations popping up in the proof of Proposition \ref{propfinale} (cf. \eqref{1}, \eqref{2}) in formula \eqref{finaledelta} to get the desired result: see Theorem \ref{finale}.

As a byproduct of our analysis, we get the following results.
\begin{enumerate}
\item {\it Uniform fomulas for $\dim\g$:}  in Proposition \ref{dim}  we get uniform formulas expressing the  dimension of $\g$ in terms of canonical data associated to the minimal grading \eqref{gradazione}.
\item {\it Application to  the Deligne exceptional series:}  in particular,  we can view the   simple Lie algebras in the Deligne exceptional series  in this framework  (cf. Remark \ref{DES}), providing  a characterization in terms of the minimal grading which yields yet another uniform derivation of the dimension formulas. 
\item {\it OPE expansions of quantum  fields of conformal weight $3/2$:} in 
Proposition \ref{propfinale}, we refine \cite[Lemma 3.1]{AKMPP} by providing a  precise  expression for the $0$-th product  in the OPE expansions of quantum  fields of conformal weight $3/2$ in $W^k(\g,\theta)$.
\end{enumerate}
\section{Yangians}
\subsection{Setup and basic relations}
Let $\g$ be a simple Lie algebra different from $sl(2)$. Fix a Cartan subalgebra $\h$  of $\g$ and  a  set $\Delta_+$  of positive roots for the $(\g,\h)$-root system $\D$. Let $\Pi$ be the corresponding set of simple roots. For $\alpha \in \Delta$ we let $\g_\a$ denote the corresponding root space. Choose a nondegenerate invariant symmetric form $(\cdot,\cdot)$ on $\g$. Denote by $\a_i,\omega_i,\theta $ the simple roots, the fundamental weights and the highest root, 
respectively. Set $\theta=\sum_in_i\a_i$. Let $\{X_\l\}_{\l\in\Lambda}$ be an orthonormal basis of $\g$.\par

As noticed in the Introduction, we will focus on the case when $\g$ is different from  $sl(2)$. We recall the definition of the Yangian in this case.
\begin{defi}[\cite{Dr1}]\label{D:YJ}
 The Yangian $Y(\g)$ is the unital associative $\C$-algebra generated by the set of elements $\{X,J(X)\,:\,X\in \g\}$ subject to the defining relations
 \begin{align}
  &XY-YX=[X,Y]_\g,  \quad J([X,Y])=[J(X),Y], \label{YJ:1}\\
  &J(cX+dY)=cJ(X)+dJ(Y),\label{YJ:2}\  \end{align}
  \begin{align}
  &[J(X),[J(Y),Z]]-[X,[J(Y),J(Z)]]=\sum_{\lambda,\mu,\nu\in \Lambda}([X,X_\la],[[Y,X_\mu],[Z,X_\nu]])\{X_\lambda,X_\mu,X_\nu\},\label{YJ:3} \end{align}
 for all $X,Y,Z,W\in \g$ and $c,d\in \C$, where $\{x_1,x_2,x_3\}=\frac{1}{24}\sum_{\pi \in \mathfrak{S}_3}x_{\pi(1)}x_{\pi(2)}x_{\pi(3)}$ for all $x_1,x_2,x_3\in Y(\g)$.

\end{defi}
\begin{rem}\label{quartanonserve}  When $\g=sl(2)$   relation \eqref{YJ:3} follows from \eqref{YJ:1} and  \eqref{YJ:2}, but  a further complicated  relation is needed: see   \cite{Dr1}, \cite[Theorem 2.6]{GRW}, \cite[3.2]{GNW} for details.
\end{rem}
%
%
%
%
%

%
%

\subsection{Drinfeld's Theorem on the minimal quantization of the adjoint representation}
In the following we provide a uniform approach to Drinfeld's description of  the minimal quantization of the adjoint representation of $Y(\g)$.\par
The following statement  sums up the content of Theorems 7 and 8 from \cite{Dr1}. 
\begin{theorem}[Drinfeld]\label{DRTEO}\ Let $\g$ be a simple Lie algebra, different from $sl(2)$. Let $\mathcal V=\g\oplus\C$. 

(a).
There exists a unique  constant $\d\in\C$ such that
the natural action of $\g$ on $\mathcal V$ extends to an action of $Y(\g)$ by setting 
\begin{equation}\label{actionY}
J(x)(y,\a)=(\d\a x, (x,y)).
\end{equation}

(b). If  either $n_i=1$ or $n_i=(\theta,\theta)/(\a_i,\a_i)$, then the fundamental representation $V_{\omega_i}$ of $\g$ extends to a $Y(\g)$-representation by letting  $J(x)$ act as $0$.
\end{theorem}
\begin{rem}
For $\g=sl(2)$ relation \ref{actionY} holds for any $\d$.
\end{rem}
To prove Theorem \ref{DRTEO} we need some preliminary work. 
Consider the maps $G_2:\bigwedge^2\g\to S^3(\g^*)$ and  $G_3:\bigwedge^3\g\to S^3(\g^*)$ defined by setting
\begin{equation}\label{G23}
G_2(X\wedge Y)(a)=([[X,a],a],[Y,a]),\ G_3(X\wedge Y\wedge Z)(a)=([[X,a],[Y,a]],[Z,a]).
\end{equation}
Let $\partial_p( X_1\wedge \ldots \wedge X_p)=\sum_{i<j} (-1)^{i+j+1} [X_i,X_j]\wedge X_1\wedge \ldots \wedge\widehat{X_i}\wedge\ldots\wedge\widehat{X_j}\wedge \ldots X_p$  be the usual boundary operator for the Lie algebra homology. The next lemma has been suggested to us by 
C. De Concini.

\begin{lemma}\label{DeC}\ 

\noindent(1)
$$
G_3=\tfrac{1}{3} G_2\circ \partial_3.
$$

\noindent(2) The maps $G_2,G_3$ are $\g$-equivariant.
\end{lemma}
\begin{proof} To prove (1) we start with the Jacobi identity:
$$
[[X,Y],a],a]=[[X,a],Y],a]+[[X,[Y,a]],a]=2[[X,a],[Y,a]]+[X,[[Y,a],a]]-[Y,[[X,a],a]],
$$
so
\begin{align*}
 &G_2\circ \partial_3(X\wedge Y\wedge Z)=6G_3(X\wedge Y\wedge Z)\\
 &+([X,[[Y,a],a]]-[Y,[[X,a],a]],[Z,a])+([Z,[[X,a],a]]-[X,[[Z,a],a]],[Y,a])\\
 &+([Y,[[Z,a],a]]-[Z,[[Y,a],a]],[X,a]).
\end{align*}
Using the invariance of the form we have
$$
G_2\circ \partial_3(X\wedge Y\wedge Z)-6G_3(X\wedge Y\wedge Z)=(Z,R(X,Y,a)),
$$
where
\begin{align*}
R(X,Y,a)&= -[[X,[[Y,a],a]],a]+[[Y,[[X,a],a]],a]+[[X,a],a],[Y,a]]\\
&+[a,[a,[X,[Y,a]]]]-[a,[a,[Y,[X,a]]]]-[[[Y,a],a],[X,a]]).
\end{align*}
Since
$$
[[X,[[Y,a],a]],a]=[[X,a],[[Y,a],a]]+[X,[[[Y,a],a],a]]
$$
and
$$
[[Y,[[X,a],a]],a]=[[Y,a],[[X,a],a]]+[Y,[[[X,a],a],a]],
$$
we can rewrite $R(X,Y,a)$ as
\begin{align*}
&-[X,[[[Y,a],a],a]]+[Y,[[[X,a],a],a]]+[a,[a,[X,[Y,a]]]]-[a,[a,[Y,[X,a]]]]\\
&=-[X,[[[Y,a],a],a]]+[Y,[[[X,a],a],a]]+[a,[a,[[X,Y],a]]]].
\end{align*}
Thus
\begin{align*}
&G_2\circ \partial_3(X\wedge Y\wedge Z)-6G_3(X\wedge Y\wedge Z)=(Z,R(X,Y,a))\\
&=(Z,-[X,[[[Y,a],a],a]]+[Y,[[[X,a],a],a]]+[a,[a,[[X,Y],a]]]])\\
&=-([[[Z,X],a],a],[Y,a])-([[[Y,Z],a],a],[X,a])-([[[X,Y],a],a],[Z,a])\\
&=-G_2\circ\partial_3(X\wedge Y\wedge Z).
\end{align*}

To prove (2)  observe that
\begin{align*}
G_2(\theta(x)(X\wedge Y)(a)&=([[x,X],a],a],[Y,a]) +(([[X,a],a],[[x,Y],a])\\
&=([x,[[X,a],a],[Y,a]) +(([[X,a],a],[x,[Y,a]])\\
&+([[X,[x,a]],a],[Y,a]) +([[X,a],[x,a]],[Y,a]) +([[X,a],a],[Y,[x,a]])\\
&=([[X,[x,a]],a],[Y,a]) +([[X,a],[x,a]],[Y,a]) +([[X,a],a],[Y,[x,a]]).
\end{align*}
Since $\partial_3$ is $\g$-equivariant, it follows from (1) that $G_3$ is $\g$-equivariant.
\end{proof}



 Identify $S^3(\g^*)$ and $S^3(\g)$ using the form $(\cdot,\cdot)$.
Set 
\begin{equation}
\phi_i=ad\circ Symm\circ G_i:\bigwedge^i\g\to End(\g),
\end{equation}
where $Symm:S(\g)\to U(\g)$ is the symmetrization map, $ad$ is the extension to $U(\g)$ of the adjoint representation $ad:\g\to End(\g)$.
If $X\in\wedge^i\g$ then, clearly, the map $\phi_i(X)(U,V)=(\phi_i(X)(U),V)$ is bilinear in $U,V$ so $\phi_i$ defines a map $g_i:\wedge^i\g\to\g^*\otimes \g^*$.

\begin{lemma}\label{inlambda21}The maps $g_i$ are alternating, thus they define maps $g_i:\wedge^i\g\to \wedge^2\g^*$.
\end{lemma}
\begin{proof}
By Lemma \ref{DeC}, in order to prove the first statement, we need only to prove  that, if $U,V\in\g$, then 
$$(g_2(X\wedge Y)(U),V)=-(g_2(X\wedge Y)(V),U).$$
Explicitly
\begin{align*}
&(g_2(X\wedge Y)(U),V)=\sum_\s\sum_{p_1,p_2,p_3}([[X,a_{p_1}],a_{p_2}],[Y,a_{p_3}])\,([a^{p_{\s(1)}},[a^{p_{\s(2)}},[a^{p_{\s(3)}},U]]],V)=\\
&-\sum_\s\sum_{p_1,p_2,p_3}([[X,a_{p_1}],a_{p_2}],[Y,a_{p_3}])\,(U,[a^{p_{\s(3)}},[a^{p_{\s(2)}},[a^{p_{\s(1)}},V]]]).
\end{align*}
Set $\tau=\s\circ(13)$; then
\begin{align*}
&\sum_\s\sum_{p_1,p_2,p_3}([[X,a_{p_1}],a_{p_2}],[Y,a_{p_3}])\,([a^{p_{\s(1)}},[a^{p_{\s(2)}},[a^{p_{\s(3)}},U]]],V)=\\
&-\sum_\tau\sum_{p_1,p_2,p_3}([[X,a_{p_1}],a_{p_2}],[Y,a_{p_3}])\,([a^{p_{\tau(1)}},[a^{p_{\tau(2)}},[a^{p_{\tau(3)}},V]]],U)
\end{align*}
as required.
\end{proof}
Extend $(\cdot,\cdot)$  to an invariant bilinear form on $\wedge^2\g$ (by determinants) and identify $\wedge^2\g^*$ with $\wedge^2\g$ using this form. In particular we can view the maps $g_i$ as maps from $\wedge^i\g$ to $\wedge^2\g$.
\begin{lemma}\label{inlambda22}
The map $g_2$ is symmetric:
$$(g_2(X\wedge Y),U\wedge V)=(X\wedge Y,g_2(U\wedge V)).
$$
\end{lemma}
\begin{proof}By unwinding all the identifications we find
\begin{align*}
(g_2(X\wedge Y),U\wedge V)&=\sum_\s\sum_{p_1,p_2,p_3}([[X,a_{p_1}],a_{p_2}],[Y,a_{p_3}])\,([a^{p_{\s(1)}},[a^{p_{\s(2)}},[a^{p_{\s(3)}},U]]],V)\\
&=\sum_\s\sum_{p_1,p_2,p_3}([a_{p_{\s^{-1}(2)}},[a_{p_{\s^{-1}(1)}},[a_{p_{\s^{-1}(3)}},X]]]],Y)\,([U,a^{p_{1}}],a^{p_{2}}],[V,a^{p_{3}}]).
\end{align*}
Set $\tau=\s^{-1}\circ(12)$; then
\begin{align*}
(g_2(X\wedge Y),U\wedge V)&=\sum_\tau\sum_{p_1,p_2,p_3}([a_{p_{\tau(1)}},[a_{p_{\tau(2)}},[a_{p_{\tau(3)}},X]]]],Y)\,([U,a^{p_{1}}],a^{p_{2}}],[V,a^{p_{3}}])\\
&=(X\wedge Y,g_2(U\wedge U)).
\end{align*}
\end{proof}

\begin{lemma}\label{costantd3}
There is a unique costant $k\in\C$ such that
\begin{equation}\label{rel3dr}
g_3=k\, \partial_3.
\end{equation}
\end{lemma}
\begin{proof}
Since  $\g\ne sl(2)$, recall that by \cite{Ko} we have orthogonal decompositions
\begin{align}
\bigwedge^2\g&= {\bf d}\g\oplus U_2,\\
\label{U2}\bigwedge^3\g&=Ker\,\partial_3\oplus Im\,{\bf d}=Ker\,\partial_3\oplus {\bf d}(U_2),
\end{align}
where $\bf d$ is the Chevalley-Eilenberg differential for Lie algebra cohomology, $U_2$ is the subspace of $\bigwedge^2\g$ generated by 2-tensors $x\wedge y$ with $[x,y]=0$. \par

Moreover, again by \cite{Ko}, $Hom_\g(\g,U_2)=0$ and $U_2$ is irreducible for $\g\ne sl(n)$, while, if $\g=sl(n)$, $U_2$ decomposes as $U_2=V_1\oplus V_2$ with $V_1$, $V_2$ inequivalent  irreducibles with $V_2=V_1^*$. 

Since
$$\phi_2=ad\circ Symm\circ G_2:\bigwedge^2 \g\to End(\g)$$ is $\g$-equivariant, by the invariance of the form, $g_2$ is also equivariant. It follows that $g_2(U_2)\subset U_2$. If $\g\ne sl(n)$, then 
\begin{equation}\label{k'}(g_2)_{|U_2}=k' I\text{ for some $k'\in \C$}.\end{equation} Note that $(Im\, \partial_3)^\perp=Ker\,{\bf d}$. Since $H^2(\g)=0$, $Ker\,{\bf d}={\bf d}\g$. It follows that $Im\,\partial_3=U_2$. Since $g_3=\tfrac{1}{3}g_2\circ \partial_3$, formula \eqref{rel3dr} is proven in this case by setting 
\begin{equation}\label{kk'}k=\tfrac{k'}{3}.\end{equation}

If $\g=sl(n)$, by the same argument, we have that $g_2(V_i)\subset V_i$, hence there is $k$ such that $(g_2)_{|V_1}=kI_{V_1}$.
Let $x\in V_2$ and $y=v_1+v_2\in U_2$ with $v_i\in V_i$. Then
$$
(g_2(x),y)=(g_2(x),v_1+v_2)=(g_2(x),v_1)=(x,g_2(v_1))=k(x,v_1)=k(x,y).
$$
Since the form is nondegenerate when restricted to $U_2$,  \eqref{rel3dr} holds in this case too.
\end{proof}

\subsubsection{Proof of Theorem \ref{DRTEO}} By \eqref{YJ:1} we must have
$$
[x,J(y)](u,0)=-(0,(y,[x,u]))=(0,([x,y],u))=J([x,y])(u,0)
$$
and
$$
[x,J(y)](0,1)=(\d[x,y],0)=J([x,y])(0,1),
$$
which holds for all $\d$.
It is clear that both sides of \eqref{YJ:3} act on $(0,1)$ trivially.

Define  $f:\g\times\g\times\g\to End(\g)$ by setting
$$
([J(X),[J(Y),Z]]-[X,[J(Y),J(Z)]])(U,0)=(f(X, Y, Z)(U),0).
$$
Then
\begin{align*}
(f(X, Y, Z)(U),W)&=\d(([Y,Z],U)(X,W)-(X,U)([Y,Z],W)-(Z,U)([X,Y],W))\\
&+\d((Y,U)([X,Z],W)+(Z,[X,U])(Y,W)-(Y,[X,U])(Z,W))\\
&=\d([X,Y]\wedge Z,U\wedge W)-([X,Z]\wedge Y,U\wedge W)+([Y,Z]\wedge X,U\wedge W))\\
&=\d(\partial_3(X\wedge Y\wedge Z),U\wedge W).
\end{align*}
We let the R.H.S. of \eqref{YJ:3} act on $(U,0)$: 
\begin{align*}
&\sum_{\lambda,\mu,\nu\in \Lambda}([X,X_\la],[[Y,X_\mu],[Z,X_\nu]])\{X_\lambda,X_\mu,X_\nu\}(U,0)\\
&=(\tfrac{1}{24}\sum_\s \sum_{p_1,p_2,p_3}([X,X_{p_1}],[[Y,X_{p_2}],[Z,X_{p_3}]])[X_{p_{\s(1)}},[X_{p_{\s(2)}},[X_{p_{\s(3)}},U],0)\\
&=(\tfrac{1}{24}\phi_3(X\wedge Y\wedge Z)(U),0),
\end{align*}
so we must have
$
\d(\partial_3(X\wedge Y\wedge Z),U\wedge W)=(\tfrac{1}{24}g_3(X\wedge Y\wedge Z),U\wedge W).
$
Thus, by Lemma \ref{costantd3}, relation \eqref{YJ:3} holds if and only if \begin{equation}\label{deltadis}\delta=\tfrac{k}{24}.\end{equation}
This proves claim (a) of the theorem. 
To prove  claim (b), set 
\begin{equation}
g^j_i=\rho_j\circ Symm\circ G_i:\bigwedge^i\g\to gl(V_{\omega_j}),
\end{equation}
where $\rho_j:\g\to gl(V_{\omega_j})$ is the j-th fundamental representation of $\g$.
Then, as shown in the next table,   $U_2$ does not appear in $V_{\omega_j}\otimes V_{\omega_j}^*$, since its highest weight $2\theta-\bar\alpha$ is not less than or equal than $\omega_j-w_0(\omega_j)$ (here $\bar\a$ is a simple root not orthogonal to $\theta$ and $w_0$ is the longest element in the Weyl group). In the exceptional cases we display the coordinates w.r.t. the choice of the simple roots from Bourbaki. 
\vskip5pt
\centerline{
\begin{tabular}{c|c|c}
Type of $\g$,  $j$ & $2\theta-\bar\alpha$ & $\omega_j-w_0(\omega_j)$\\\hline
$A_n,\,1\leq j\leq [n+1]/2$&$\epsilon_1+\epsilon_2-2\epsilon_{n+1}, 2\epsilon_1-\epsilon_n-\epsilon_{n+1}$&$\sum_{h=1}^j(\epsilon_h-\epsilon_{n+2-h})$\\ \hline
$B_n,\,j=1$&$2\epsilon_1+\epsilon_2+\epsilon_3$&$2\epsilon_1$\\\hline
$B_n,\,j=n$&$2\epsilon_1+\epsilon_2+\epsilon_3$&$\sum_{i=1}^n\epsilon_i$\\\hline
$C_n,\,1\leq j\leq n$&$2\epsilon_1+\epsilon_2$&$\sum_{i=1}^j\epsilon_i$\\\hline
$D_n,\,j=1$&$2\epsilon_1+\epsilon_2+\epsilon_3$&$2\epsilon_1$\\\hline
$D_n,\,j=n-1\,(n) $&$2\epsilon_1+\epsilon_2+\epsilon_3$&$\sum_{i=1}^{n-1}\epsilon_i$\\\hline
$E_6,\,j=1\,\, (6 )$&$(2,3,4,6,4,2)$&$(2,2,3,4,3,2)$\\\hline
$E_7,\,j=7$&$(3,4,6,8,6,4,2)$&$(2,3,4,6,5,4,3)$\\\hline
$F_4,\,j=4$&$(3,6,8,4)$&$(1,2,3,2)$\\\hline
$G_2,\,j=1$&$(6,3)$&$(2,1)$
\end{tabular}
}

\section{Minimal \texorpdfstring{$\tfrac{1}{2}\mathbb Z$}{-}-grading of a simple Lie algebra}\label{mingrad}

Choose root vectors $x_{\pm\theta}\in\g_{\pm\theta}$ so that $(x_\theta|x_{-\theta})=\tfrac{1}{2}$.
Set  $x=[x_\theta,x_{-\theta}]$. The eigenspace decomposition of $ad\,x$ defines the {\it minimal} $\frac{1}{2}\ganz$-grading:
\begin{equation}\label{gradazione}
\g=\g_{-1}\oplus\g_{-1/2}\oplus\g_{0}\oplus\g_{1/2}\oplus\g_{1},
\end{equation}
where $\g_{\pm 1}=\C  x_{\pm \theta}$.    Furthermore, one has
\begin{equation}\label{gnatural}
\g_0=\g^\natural\oplus \C x,\quad\g^\natural=\{a\in\g_0\mid (a|x)=0\}.
\end{equation}
Note that  $\g^\natural$ is the centralizer of the triple $\{x_{-\theta},x,x_{\theta}\}$.
We can choose $
\h^\natural=\{h\in\h\mid (h|x)=0\}
$ as a  Cartan subalgebra of $\g^\natural$,  so that $\h=\h^\natural\oplus \C x$.
Set, for $u,v\in \g_{-1/2}$,
$$
\langle u,v\rangle=(x_\theta|[u,v])
$$
and note that $\langle\cdot,\cdot\rangle$ is a $\g^\natural$--invariant symplectic form on $\g_{-1/2}$. 

We will use the following terminology.
\begin{defi} We say that an ideal in $\g^\natural$ is {\it irreducible} if it is simple or $1$-dimensional. We call such an ideal a {\it component} of $\g^\natural$.
\end{defi}

Write $$\g^\natural=\bigoplus_{i=1}^r\g^\natural_i$$ with $\g^\natural_i$ irreducible.  Recall that $r=1,2$ or $3$. For a simple Lie algebra $\aa$ we let $h^\vee_\aa$ to be its dual Coxeter number and, if $\aa$ is abelian, we set $h^\vee_\aa=0$. Set $h^\vee=h^\vee_\g$ and $h^\vee_i=h^\vee_{\g^\natural_i}$.
Let $\nu_i$ be the ratio of  the 
normalized invariant form of $\g$ restricted to $\g_i^\natural$ and the normalized form on $\g_i^\natural$. Set finally $\bar h_i^\vee=  h_i^\vee/\nu_i$.
For  reader's convenience, we display the  relevant data  in the following Table (although we proceed uniformly, so we do not need to use them).
\vskip 5pt
\centerline{\begin{tabular}{c|c|c|c|c}
$\g$&$\g^\natural$&$\g_{1/2}$&$h^\vee$&$\bar h_i^\vee$\\
\hline
$sl(3)$&$\C$&$\C\oplus \C^* $&$3$&$0$\\\hline 
$sl(n), n\geq 4$&$gl(n-2)$&$\C^{n-2}\oplus (\C^{n-2})^* $&$n$&$0,n-2$\\\hline 
$so(n), n>6, n\ne 8$&$sl(2)\oplus so(n-4)$&$\C^2\otimes\C^{n-4}$&$n-2$&$2,n-6$\\\hline
$so(8)$&$sl(2)\oplus sl(2)\oplus sl(2)$&$\C^2\otimes\C^2\otimes \C^2$&$6$&$2,2,2$\\\hline
$sp(2n), n\geq 2$&$sp(2n-2)$&$\C^{2n-2} $&$n+1$&$n$\\\hline
$G_2$&$sl(2)$&$S^3\C^2$&$4$&$4/3$\\\hline
$F_4$&$sp(6)$&$\bigwedge_0^3\C^6$ & $9$&$4$\\\hline
$E_6$&$sl(6)$&$\bigwedge^3\C^6$ & $12$&$6$\\\hline
$E_7$&$so(12)$&$spin_{12}$ & $18$&$10$\\\hline
$E_8$&$E_7$&$\dim=56$ & $30$&$18$
\end{tabular}}
\vskip10pt

Consider now the involution $\s_x=e^{2\pi \sqrt{-1} ad(x)}$. Since $\a_i(x)\ge 0$ for all simple roots $\a_i$ and $\theta(x)=1$, it follows that the set $\{1-\theta(x)\}\cup\{\a_i(x)\mid \a_i \text{ simple root}\}$ is the set of Kac parameters for the automorphism $\s_x$. In particular, since $\s_x$ is an involution, either there is a unique simple root $\a_{i_0}$ such that $\a_{i_0}(x)\ne 0$ or there are exactly two simple roots $\a_{i_0},\a_{i_1}$ such that $\a_{i_j}(x)\ne 0$. Let $s$ be the number of simple roots not orthogonal to $\theta$. If $s=1$ then $\a_{i_0}(x)= \tfrac{1}{2}$ and $n_{i_0}=2$. 
If $s=2$, then   $\a_{i_0}(x)=\a_{i_1}(x)= \tfrac{1}{2}$,  $n_{i_0}=n_{i_1}=1$.
\vskip5pt
Write $\g=\k\oplus\p$ for the eigenspace decomposition of $\s_x$.  We observe that
$$
\k=span(x_\theta,x,x_{-\theta})\oplus \g^\natural\simeq sl(2)\times \g^\natural,\quad \p=\g_{1/2}\oplus \g_{-1/2}.
$$
One can choose the set of positive roots for $\k$ so that the corresponding set of simple roots is $\{-\theta\}\cup \{\a\in\Pi\mid \a(x)=0\}= \{-\theta\}\cup \{\a\in\Pi\mid (\a|\theta)=0\}$.

Consider the case $s=1$. Then  $\g^\natural$ is semisimple and the number of simple ideals of $\g^\natural$ equals the number of roots attached to $\a_{i_0}$. Moreover $\p$ is irreducible and its highest weight as a $\k$-module is $-\a_{i_0}$. Since $\a_{i_0}(\theta^\vee)=1$, we see that 
$\p=V_{sl(2)}(\omega_1)\otimes V_{\g^\natural}(-(\a_{i_0})_{|\h^\natural})$. (Here $V_\aa(\lambda)$ denotes the irreducible finite dimensional $\aa$-module of highest weight $\l$). If $U$ is a $ad(x)$-stable space we let $U_k$ denote the eigenspace corresponding to the eigenvalue $k$.
Since $V_{sl(2)}(\omega_1)=V_{sl(2)}(\omega_1)_{1/2}\oplus V_{sl(2)}(\omega_1)_{-1/2}$ we see that, as $\g^\natural$-module,
$$
\g_{1/2}\simeq \g_{-1/2}\simeq V_{\g^\natural}(-(\a_{i_0})_{|\h^\natural}).$$
In particular $\g_{\pm1/2}$ are irreducible as $\g^\natural$-modules.

If $s=2$ we have $\g^\natural_0=\C\varpi$ with 
\begin{equation}\label{varpi}\varpi=\omega_{i_0}^\vee-\omega_{i_1}^\vee.\end{equation} Moreover $\p=V_\k(-\a_{i_0})\oplus V_{\k}(-\a_{i_1})$. Arguing as above we obtain that 
$$
\g_{1/2}\simeq \g_{-1/2}\simeq V_{\g^\natural}(-(\a_{i_0})_{|\h^\natural})\oplus V_{\g^\natural}(-(\a_{i_1})_{|\h^\natural}).$$
Since $\varpi$ acts on $V_{\g^\natural}(-(\a_{i_j})_{|\h^\natural})$ as $-(-1)^jI$ we see that $\g_{-1/2}$ is the sum of two inequivalent $\g^\natural$-modules.

As shown in \cite[Proposition 4.8]{CMPP}, $V_\k(-\a_{i_0})^*\simeq V_{\k}(-\a_{i_1})$, hence, 
$$V_{\g^\natural}(-(\a_{i_0})_{|\h^\natural})^*\simeq V_{\g^\natural}(-(\a_{i_1})_{|\h^\natural}).$$

We now turn to the study of $\wedge^2\g_{-1/2}$ as a $\g^\natural$-module.
 Let ${\bf d}, {\bf d}_\k$ be coboundary operators for the Lie algebra cohomology of $\g$, $\k$ respectively and set ${\bf d}_1=
 {\bf d}-{\bf d}_\k$. 
 By \cite[Prop. 4.3]{P}, $\wedge^2\p={\bf d}_1\k\oplus V'$, where
 \begin{equation}\label{v'}
 V'=span(u\wedge v\mid u,v\in\p,\ [u,v]=0).
 \end{equation} We observe that $\wedge^2\g_{-1/2}=(\wedge^2\p)_{-1}$, so 
 $$
 \wedge^2\g_{-1/2}={\bf d}_1(\k)_{-1}\oplus V'_{-1}.
 $$
 As $\k=span(x_\theta,x,x_{-\theta})\oplus \g^\natural$, we see that ${\bf d}_1(\k)_{-1}=\C {\bf d}_1( x_{-\theta})\simeq V_{\g^\natural}(0)$. Set $\Pi_j=\{\a\in \Pi\mid (\a,\a_{i_j})\ne 0\}\cup\{-\theta\}$, $\mathcal W'_j=\{-\a_{i_j}-s_{\a_{i_j}}(\a)\mid \a\in\Pi_j\}$, and 
 $$\mathcal W'=\begin{cases}\mathcal W'_0\quad&\text{if $s=1$,}\\
 \mathcal W'_0\cup \mathcal W'_1\quad&\text{if $s=2$ and $(\a_{i_0},\a_{i_1})\ne 0$,}\\
 \mathcal W'_0\cup \mathcal W'_1\cup \{-\a_{i_0}-\a_{i_1}\}\quad&\text{if $s=2$ and $(\a_{i_0},\a_{i_1})=0$.}
 \end{cases}
 $$ Recall from \cite{CMP} that 
$$
 V'=\oplus_{\l\in\mathcal W'}V_\k(\l)
$$
 and that, if $\l=-\a_{i_j}-s_{\a_{i_j}}(\a)$, then the highest weight vector is $x_{-\a_{i_j}}\wedge x_{-s_{\a_{i_j}}(\a)}$. 
 Set explicitly $\{\l_1,\ldots,\l_p\}=\{\l\in\mathcal W'\mid \l(x)=-1\}$, so that,  the $(-1)$-eigenspace of $ad(x)$ is
 $$
V'_{-1}=\bigoplus_{i=1}^pV_{\g^\natural}({\l_i}_{|\h^\natural}).
 $$
 
 Consider the map $\Phi:\bigwedge^2 \g_{-1/2}\to S^2((\g^\natural)^*)$ defined by 
\begin{equation}\label{Phi}
\Phi(u\wedge v)(a)=\langle[u,a],[v,a]\rangle.
\end{equation}
It is easy to check that the map $\Phi$ is $\g^\natural$-equivariant. 
\begin{prop}\label{Phinontrivial} Let $P: \bigwedge^2 \g_{-1/2}\to S^2((\g^\natural)^*)$ be a  $\g^\natural$-equivariant map. Then there are constants $f_1,\ldots,f_p$ such that 
$$P_{|V_{\g^\natural}(\l_i)}=f_i \ \Phi_{|V_{\g^\natural}(\l_i)}$$
for $1\leq i\leq p$.\end{prop}
\begin{proof} By \cite[Proposition 2.1]{R}, $S^2(\g^\natural)$ decomposes with multiplicity one. Since the same happens to $V'$ (cf. \eqref{v'}), it suffices to prove that if $P_{|V_{\g^\natural}(\l_i)}$ is nonzero then 
$\Phi_{|V_{\g^\natural}(\l_i)}$ is nonzero.

Assume that $\g^\natural$ is semisimple. In this case we prove that $Ker \Phi=\{0\}$. It is enough to check that $\Phi(x_{-\a_{i_j}}\wedge x_{-s_{\a_{i_j}}(\a)})\ne 0$ for all $\a\in \Pi_j$ such that $(\a_{i_j}+s_{\a_{i_j}}(\a))(x)=1$. Since $\a_{i_j}(x)=1/2$ and $\a(x)=0$, we see that $s_{\a_{i_j}}(\a)=\a+\a_{i_j}$.
It follows that
\begin{equation}\label{a}
\Phi(x_{-\a_{i_j}}\wedge x_{-s_{\a_{i_j}}(\a)})(a)=\langle[x_{-\a_{i_j}},a],[x_{-\a_{i_j}-\a},a]\rangle.
\end{equation}

Assume first that there is a unique simple root $\a_{i_0}$ not orthogonal to $\theta$. 
If  $\a_{i_0}$ is a short root then $x_{-\theta+2\a_{i_0}}\in\g^\natural$. 
Since $-\theta+\a_{i_0}-\a,-\a_{i_0}+\a$ are not  roots, we obtain, taking $a=x_{-\theta+2\a_{i_0}}+x_{\a}$,
\begin{align*}
\langle[x_{-\a_{i_0}},a],[x_{-\a_{i_0}-\a},a]\rangle&=\langle[x_{-\a_{i_0}},x_{-\theta+2\a_{i_0}}],[x_{-\a_{i_0}-\a},x_{\a}]\rangle\ne0.
\end{align*}

We can therefore assume that $\a_{i_0}$ is a long root. 
Assume  $\a$ short. The fact that $\a_{i_0}$ is long implies that $(\a|\a_{i_0})=-1$, so 
$$
(\theta-\a_{i_0}-\a|\a)=(-\a_{i_0}-\a|\a)=1-(\a|\a)\ge0.
$$
Since $\theta-\a_{i_0}=\theta-\a_{i_0}-\a+\a$ is a root, it follows that $\theta-\a_{i_0}-2\a$ is a positive root. Since
$$
(\theta-\a_{i_0}-2\a|\a_{i_0})=1-2+2=1,
$$
$\theta-2\a_{i_0}-2\a$ is a positive root as well.
We choose $a=x_{-\a}+x_{-\theta+2\a_{i_0}+2\a}$, so 
\begin{align*}
\langle[x_{-\a_{i_0}},a],[x_{-\a_{i_0}-\a},a]\rangle&=\langle[x_{-\a_{i_0}},x_{-\a}],[x_{-\a_{i_0}-\a},x_{-\theta+2\a_{i_0}+2\a}]\rangle\\
&+\langle[x_{-\a_{i_0}},x_{-\theta+2\a_{i_0}+2\a}],[x_{-\a_{i_0}-\a},x_{-\a}]\rangle\\
&=\langle[x_{-\a_{i_0}},x_{-\a}],[x_{-\a_{i_0}-\a},x_{-\theta+2\a_{i_0}+2\a}]\rangle\\
&+\langle x_{-\a_{i_0}},[x_{-\theta+2\a_{i_0}+2\a},[x_{-\a_{i_0}-\a},x_{-\a}]]\rangle\\
&=2\langle[x_{-\a_{i_0}},x_{-\a}],[x_{-\a_{i_0}-\a},x_{-\theta+2\a_{i_0}+2\a}]\rangle\\
&+\langle x_{-\a_{i_0}},[x_{-\a_{i_0}-\a},[x_{-\theta+2\a_{i_0}+2\a},x_{-\a}]]\rangle.
\end{align*}
If  $-\theta+2\a_{i_0}+\a$ is a root, then
$$
(-\theta+2\a_{i_0}+\a|\a_{i_0})=-1+4-1=2,
$$
contradicting the fact that $\a_{i_0}$ is long. It follows that $[x_{-\theta+2\a_{i_0}+2\a},x_{-\a}]=0$, hence
$$
\langle[x_{-\a_{i_0}},a],[x_{-\a_{i_0}-\a},a]\rangle=2\langle[x_{-\a_{i_0}},x_{-\a}],[x_{-\a_{i_0}-\a},x_{-\theta+2\a_{i_0}+2\a}]\rangle\ne0.
$$

We can therefore assume that $\a$ is long. Let $\Sigma$ be the component of the Dynkin diagram of $\g^\natural$ containing $\a$. 
Let $\theta_\Sigma$ be its highest root.
If $\a=\theta_\Sigma$, then $\g^\natural$ is not simple, for, otherwise, $\g\simeq sl(3)$. Let $\Sigma'$ be the Dynkin diagram of another component of $\g^\natural$. Then $\a+\a_{i_0}+\theta_{\Sigma'}$ is a positive root as well as $\theta-\a-\a_{i_0}-\theta_{\Sigma'}$. Since $\a_{i_0}$ is a long root, $(\theta_{\Sigma'}|\a_{i_0})=-1$, hence
$$
(\theta-\a-\a_{i_0}-\theta_{\Sigma'}|\a_{i_0})=1+1-2+1=1,
$$
so that $\theta-\a-2\a_{i_0}-\theta_{\Sigma'}$ is a positive root. We choose $a=x_{-\theta_{\Sigma'}}+x_{-\theta+\a+2\a_{i_0}+\theta_{\Sigma'}}$.

Arguing as above, we conclude that 
$$
\langle[x_{-\a_{i_0}},a],[x_{-\a_{i_0}-\a},a]\rangle=2\langle[x_{-\a_{i_0}},x_{-\theta_{\Sigma'}}],[x_{-\a_{i_0}-\a},x_{-\theta+\a+2\a_{i_0}+\theta_{\Sigma'}}]\rangle\ne0.
$$

We can therefore assume both $\a$ and $\a_{i_0}$ long roots and $\a\ne\theta_\Sigma$.
Since $(\a_{i_0}|\a)<0$, we have that $\a_{i_0}+\theta_\Sigma$ is a root. Since $(\theta|\a_{i_0}+\theta_\Sigma)=(\theta|\a_{i_0})>0$, we see that $\theta-\a_{i_0}-\theta_\Sigma$ is a positive root.

Since $\a_{i_0}$ is long, by Lemma 5.7 of \cite{CMP}, $\a$ has coefficient $1$ in $\theta_\Sigma$.  We now prove that \begin{equation}\label{ee} (\theta_\Sigma,\a)=0.\end{equation}
Indeed, if \eqref{ee} holds, $(\theta-\a_{i_0}-\theta_\Sigma,\a)=1$ so  $\theta-\a_{i_0}-\a-\theta_\Sigma$ is a root. Since 
 $(\theta-\a_{i_j}-\theta_\Sigma-\a, \a_{i_0})=1-2+1+1=1>0$, we obtain that $\theta-2\a_{i_0}-\theta_\Sigma-\a$ is a positive root. Choosing $a=x_{-\theta_{\Sigma}}+x_{-\theta+\a+2\a_{i_0}+\theta_{\Sigma}}$ and arguing as above, we conclude that 
$$
\langle[x_{-\a_{i_0}},a],[x_{-\a_{i_0}-\a},a]\rangle=2\langle[x_{-\a_{i_0}},x_{-\theta_{\Sigma}}],[x_{-\a_{i_0}-\a},x_{-\theta+\a+2\a_{i_0}+\theta_{\Sigma}}]\rangle\ne0.
$$

 To prove \eqref{ee} we need to use definitions and results from \cite{CMPP}.  If \eqref{ee} does not hold, then $\a\notin A(\Sigma)$
 (see \cite[Definition 4.1]{CMPP} for the definition of $A(\Sigma)$). Then $A(\Sigma)=\widehat \Pi\setminus \Sigma$ and by  \cite[Proposition 4.2]{CMPP}, $\theta-\theta_\Sigma$  is supported outside  $\Sigma$ so
 $$\theta=\sum_{\eta\notin \Sigma, \eta\ne \a_{i_0}}n_\eta \eta +2\a_{i_0}+\theta_\Sigma.$$
This contradicts the fact that $\theta-\a-\a_{i_0}$ is a root since $\a\ne\theta_\Sigma$ and the coefficient of $\a$ in $\theta_\Sigma$ is $1$.

We are left with the case when $\g^\natural$ ha a 1-dimensional center generated by $\varpi$ (cf. \eqref{varpi}). We can assume that $(\a_{i_0}|\a_{i_1})=0$. If  this is not the case, then $\g=sl(3)$ and $V'=\{0\}$.
Suppose  $\sigma$ is a  simple roots not orthogonal  to $\a_{i_0}$. Then $(-2\a_{i_0}-\sigma)(\varpi)=-2$. In particular $V_{\g^\natural}((-2\a_{i_0}-\sigma)_{|\h^\natural})$ does not occur in $S^2(\g^\natural)$.
It follows that $P_{|V_{\g^\natural}((-2\a_{i_0}-\sigma)_{|\h^\natural})}=\Phi_{|V_{\g^\natural}((-2\a_{i_0}-\sigma)_{|\h^\natural})}=0$. Note that the highest weight vector of 
It remains only to check that $\Phi_{|V_{\g^\natural}((-\a_{i_0}-\a_{i_1})_{|\h^\natural}))}$ is nonzero. 
$V_{\g^\natural}((-\a_{i_0}-\a_{i_1}))$ is $x_{-\a_{i_0}}\wedge x_{-\a_{i_1}}$. We need to find $a$ such that 
$$\langle [x_{-\a_{i_0}},a],[x_{-\a_{i_1}},a]\rangle\ne 0.$$
Let $\Sigma$ be a  component of $\g^\natural$ attached to $\a_{i_0}$, so that $\a_{i_0}+\theta_\Sigma, \theta-\a_{i_0}-\theta_\Sigma$ are both roots.  Since $(\theta_\Sigma|\a_{i_1})\le0$, hence
$(\theta-\a_{i_0}-\theta_\Sigma|\a_{i_1})>0$ and in turn $\theta-\a_{i_0}-\theta_\Sigma-\a_{i_1}$ is a positive root or zero. Choosing $a=x_{-\theta_\Sigma}+x_{-\theta+\a_{i_0}+\theta_\Sigma+\a_{i_1}}$ in the first case and $a=x_{-\theta_\Sigma}+h_{\theta_\Sigma}$ in the second we are done.
 \end{proof}
\section{Calculation of \texorpdfstring{$\delta$}{d}}

 We note that for computing $\d$  it is enough to compute 
\begin{equation}\label{enough}
(g_2(a\wedge b),c\wedge d),\qquad a,b,c,d\in\g,
\end{equation}
where  $[a,b]=[c,d]=0$ and $(a\wedge b,c\wedge d)\ne0$.

Indeed, recall from \eqref{k'} that there is a constant $k'$ such that $g_2(a\wedge b)=k' a\wedge b$.
By  \eqref{kk'} and \eqref{deltadis} we obtain
\begin{equation}\label{howc}
\delta=\tfrac{k'}{72}=\frac{(g_2(a\wedge b),c\wedge d)}{72(a\wedge b,c\wedge d)}.
\end{equation}
We first use \eqref{howc} to determine the dependence of $\delta$ from the choice of the form $(\cdot,\cdot)$. Let us write $\d$ as $\delta^{(\cdot,\cdot)}$ to emphasize this dependence.
\begin{lemma}\label{depform}
$$\delta^{s(\cdot,\cdot)}=\frac{1}{s^3}\delta^{(\cdot,\cdot)}.
$$
\end{lemma}
\begin{proof}
Let us write $g_2$ as $g_2^{(\cdot,\cdot)}$ to emphasize the dependence on $(\cdot,\cdot)$.
Then,  
$$
s^2(g_2^{s(\cdot,\cdot)}(a\wedge b),c\wedge d)=\sum_{t_1,t_2,t_3,\s}s([[a,x_{t_1}],x_{t_2}],[b,x_{t_3}])s([\tfrac{x^{t_{\s(1)}}}{s},[\tfrac{x^{t_{\s(2)}}}{s},[\tfrac{x^{t_{\s(3)}}}{s},c]]],d)
$$
hence
$$
(g_2^{s(\cdot,\cdot)}(a\wedge b),c\wedge d)=\frac{1}{s^3}(g_2^{(\cdot,\cdot)}(a\wedge b),c\wedge d).
$$
It follows that
$$
\delta^{s(\cdot,\cdot)}=\frac{s^2(g_2^{s(\cdot,\cdot)}(a\wedge b),c\wedge d)}{72s^2(a\wedge b,c\wedge d)}=\frac{1}{s^3}\frac{(g_2^{(\cdot,\cdot)}(a\wedge b),c\wedge d)}{72(a\wedge b,c\wedge d)}=\frac{1}{s^3}\d^{(\cdot,\cdot)}.
$$
\end{proof}

Lemma \ref{depform} allows us to choose   as invariant form   the normalized one, that we denote by $(\cdot|\cdot)$. Recall that the normalized invariant form is defined by setting $(\theta|\theta)=2$. From now on by $\d$ we mean $\d^{(\cdot|\cdot)}$.

Choose a simple root $\a_{i_0}$ such that $(\theta|\a_{i_0})\ne 0$ ($\a_{i_0}$  is unique up to type $A$).
 Set $\gamma=\theta-\a_{i_0}$. Choose root vectors $e_{i_0}, f_{i_0}$,  in $\g_{\pm\a_{i_0}}$
such that $(e_{i_0}| f_{i_0})=\frac{2}{(\a_{i_0},\a_{i_0})}$ and set $x_{\gamma}=[x_\theta,f_{i_0}]$, $x_ {-\gamma}=[x_{-\theta},e_{i_0}]$.

We compute $\delta$ by specializing \eqref{howc} to the case where $(\cdot,\cdot)=(\cdot|\cdot)$ and $a=x_{-\theta},b=x_{-\gamma}, c=x_{\theta}, d=x_{\gamma}$ so that
\begin{equation}\label{g}
\delta=\frac{(g_2(x_{-\theta}\wedge x_{-\gamma})|x_{\theta}\wedge x_ {\gamma})}{72(x_{-\theta}\wedge x_{-\gamma}|x_{\theta}\wedge x_ {\gamma})}=\frac{(g_2(x_{-\theta}\wedge x_{-\gamma})|x_{\theta}\wedge x_ {\gamma})}{36(x_{-\gamma}|x_{\gamma})}.
\end{equation}

We choose a basis $\{u_i\}$ of $\g_{-1/2}$ and let $\{u^i\}$ be its dual basis (i. e. $\langle u_i,u^j\rangle=\delta_{ij}$). We also choose an orthonormal basis $\{x_i\}$ of $\g^\natural$. Then, as basis of $\g$, we can choose 
\begin{equation}\label{basisg}
\{x_\theta\}\cup\{[x_\theta,u_i]\}\cup\{x_i\}\cup\{x\}\cup\{u_i\}\cup\{x_{-\theta}\}.
\end{equation}
The corresponding dual basis (w.r.t. $(\cdot|\cdot)$) is
\begin{equation}\label{dualbasisg}
\{2x_{-\theta}\}\cup\{u^i\}\cup\{x_i\}\cup\{2x\}\cup\{-[x_\theta,u^i]\}\cup\{2x_\theta\}.
\end{equation}
We choose an orthonormal basis $\{x_i^r\}$ for each component $\g^\natural_r$ so that we can set $\{x_i\}=\cup_r\{x_i^r\}$.

Let $C_{\g^\natural}=\sum_i (x_i)^2$ be the Casimir element of $\g^\natural$ and $C_{\g_0}$ the Casimir element of $\g_0$. Since $C_{\g^\natural}=C_{\g_0}-2x^2$ by Lemma 5.1 of \cite{KW} we have  that
\begin{equation}
\sum_i[x_i,[x_i,v]]=(h^\vee-\tfrac{3}{2})v.\label{casgm}
\end{equation}
Recall also that it follows from  Lemma 5.1 of \cite{KW} that
\begin{equation}
\dim \g_{-1/2}=\dim \g_{1/2}=2h^\vee-4.\label{dimgm}
\end{equation}
We  extend $\langle\cdot,\cdot\rangle$ on $\wedge^2\g_{-1/2}$ by determinants:
$$
\langle u\wedge v,w\wedge z\rangle=\langle u,w\rangle\langle v,z\rangle-\langle u,z\rangle\langle v,w\rangle.
$$
We collect various formulas in the following lemma.
\begin{lemma}\label{lemma} If $u,v,w,z\in \g_{-1/2}$ and $a\in \g^\natural_r$, then 
\begin{align}
\sum_{p}[[x^s_p,v],[a,x^s_p]]&=\d_{rs}(\bar h_r^\vee)[v,a],\label{F1}\\
\sum_{p}[x^s_p,[v,[a,x^s_p]]]&=-\d_{rs}(\bar h_r^\vee)[v,a],\label{F2}\\
\sum_{p}[x_p,[a,[v,x_p]]]&=(h^\vee-3/2-\bar h_r^\vee)[v,a],\label{F3}\\
[[x_\theta,u],[x_\theta,v]]&=-\langle u,v\rangle x_\theta,\label{F4bis}\\
[u,v]&=2\langle u,v\rangle x_{-\theta},\label{F4}\\
[[x_\theta,u],v]&=\sum_i\langle u,[v,x_i]\rangle x_i+\langle u,v\rangle x,\label{F5}\\
[v,[x_\theta,u]]&=\sum_i\langle [v,x_i],u\rangle x_i+\langle v,u\rangle x,\label{F6}\\
\sum_i[u,x_i]\wedge [v,x_i]&=-\frac{\langle u,v\rangle}{2}\sum_{r} u_r\wedge u^r-\tfrac{1}{2}u\wedge v,\label{F8}\\
\sum_{i,j}\langle [u,x_i],[v,x_j]&\rangle\langle[w,x_i],[z,x_j]\rangle-\sum_{i,j}\langle [w,x_i],[v,x_j]\rangle\langle[u,x_i],[z,x_j]\rangle\notag\\
&=
(h^\vee-1)\langle u,w\rangle\langle v,z\rangle+\tfrac{1}{4}\langle u\wedge w,v\wedge z\rangle,\label{F7}
\end{align}
    \end{lemma}
\begin{proof}If $r\ne s$, then \eqref{F1} and \eqref{F2} are obvious. If $r=s$ then, on one hand,
\begin{align*}
\sum_{p}[x^r_p,[v,[a,x^r_p]]]&=\sum_{p,j}[x^r_p,[v,([a,x^r_p]|x_j^r)x_j^r]]=-\sum_{p,j}[x^r_p,[v,(x^r_p|[a,x_j^r])x_j^r]]=\\
&=-\sum_{j}[[a,x_j^r],[v,x_j^r]]=-\sum_{p}[[x^r_p,v],[a,x^r_p]].
\end{align*}
On the other hand 
\begin{align*}
\sum_{p}[x^r_p,[v,[a,x^r_p]]]&=\sum_{p}[[x^r_p,v],[a,x^r_p]]+\sum_{p}[v,[x^r_p,[a,x^r_p]]]\\&=\sum_{p}[[x^r_p,v],[a,x^r_p]]-2(\bar h_r^\vee)[v,a],
\end{align*}
so
$$
\sum_{p}[[x^r_p,v],[a,x^r_p]]-2(\bar h_r^\vee)[v,a]=-\sum_{p}[[x^r_p,v],[a,x^r_p]],
$$
hence
$$
\sum_{p}[[x^r_p,v],[a,x^r_p]]=(\bar h_r^\vee)[v,a],
\quad
\sum_{p}[x^r_p,[v,[a,x^r_p]]]=-(\bar h_r^\vee)[v,a],
$$
and
$$
\sum_{p}[x_p,[a,[v,x_p]]]=\sum_{p}[x_p,[[a,v],x_p]]]+\sum_{p}[x^r_p,[v,[a,x^r_p]]]=(h^\vee-3/2-\bar h_r^\vee)[v,a].
$$
This proves \eqref{F1}, \eqref{F2}, and \eqref{F3}.
Formulas \eqref{F4bis}, \eqref{F4}, \eqref{F5}, \eqref{F6} are straightforward.

We now prove \eqref{F8}. Note
that
$\g_{-1/2}\wedge \g_{-1/2}=\C\sum_i u_i\wedge u^i\oplus V_2$ with $V_2=span(u\wedge v\mid [u,v]=0)$. If $u,v\in\g_{-1/2}$, then the corresponding decomposition of $u\wedge v$  is
$$
u\wedge v=\frac{\langle u,v\rangle}{\dim \g_{1/2}}\sum_i u_i\wedge u^i+s, \quad s\in V_2.
$$
If $w,z\in\g_{-1/2}$ then
$$
\sum_n [w,x_n]\wedge [z,x_n]=\tfrac{1}{2}(C_{\g^\natural}(w\wedge z)-C_{\g^\natural}w\wedge z-w\wedge C_{\g^\natural}z).
$$

 Assume now 
$[w, z]=0$ and set $C_{sl(2)}=2(x^2+x_{\theta}x_{-\theta}+x_{-\theta}x_{\theta})$ to be the Casimir element of $span(x_\theta,x,x_{-\theta})\simeq sl(2)$.
By \cite{P},   $(C_{sl(2)}+C_{\g^\natural})(w\wedge z)=2h^\vee (w\wedge z)$. 
Since
$$
C_{sl(2)}(w\wedge z)=C_{sl(2)}w\wedge z+w\wedge C_{sl(2)}z+4xw\wedge xz,
$$
$C_{sl(2)}w=\tfrac{3}{2}w$, and $4xw\wedge xz=w\wedge z$, we obtain
$C_{sl(2)}(w\wedge z)=4w\wedge z$ so that  
$
C_{\g^\natural}(w\wedge z)=(2h^\vee-4)w\wedge z
$. Using  \eqref{casgm},  the final outcome is that
$$
\sum_n [w,x_n]\wedge [z,x_n]=\tfrac{1}{2}(2h^\vee-4-2(h^\vee-3/2))(w\wedge z)=-\tfrac{1}{2}(w\wedge z),
$$
By \eqref{dimgm}, \eqref{casgm}, and the above formula applied to $s$, we find
$$
\sum_i[u,x_i]\wedge [v,x_i]=-\frac{\langle u,v\rangle}{2}\frac{2h^\vee-3}{2h^\vee-4}\sum_{r} u_r\wedge u^r-\tfrac{1}{2}(u\wedge v-\frac{\langle u,v\rangle}{2h^\vee-4}\sum_i u_i\wedge u^i)
$$
hence \eqref{F8}.

Finally, we prove \eqref{F7}:
\begin{align*}
&\sum_i\langle[u,x_i]\wedge [w,x_i],v\wedge z)=-\frac{\langle u,w\rangle}{2}\sum_r(\langle u_r,v\rangle\langle u^r,z\rangle-\langle u^r,v\rangle\langle u_r,z\rangle)-\tfrac{1}{2}\langle u\wedge w,v\wedge z\rangle\\
&=-\langle u,w\rangle\langle v,z\rangle-\tfrac{1}{2}\langle u\wedge w,v\wedge z\rangle.
\end{align*}
It follows that
\begin{align*}
&\sum_{i,j}\langle [u,x_i],[v,x_j]\rangle\langle[w,x_i],[z,x_j]\rangle-\sum_{i,j}\langle [w,x_i],[v,x_j]\rangle\langle[u,x_i],[z,x_j]\rangle\\
&=
-\sum_j(\langle u,w\rangle\langle [v,x_j],[z,x_j]\rangle-\tfrac{1}{2}\sum_j\langle u\wedge w,[v,x_j]\wedge [z,x_j]\rangle\\
&=
(h^\vee-3/2)\langle u,w\rangle\langle v,z\rangle-\tfrac{1}{2}\sum_j\langle u\wedge w,[v,x_j]\wedge[z,x_j]\rangle\\
&=
(h^\vee-1)\langle u,w\rangle\langle v,z\rangle+\tfrac{1}{4}\langle u\wedge w,v\wedge z\rangle,
\end{align*}
which is precisely \eqref{F7}.
\end{proof}

We choose the basis $\{a_i\}$ for $\g$ to be the basis displayed in \eqref{basisg}. Then the dual basis   $\{a^i\}$  is the basis given in  \eqref{dualbasisg}.
We start our computation of $\d$ by computing  
\begin{equation}\label{G2}
G_2(x_{-\theta}\wedge x_{-\gamma})=\sum_{i,j,k}([[x_{-\theta},a_i],a_j]|[x_{-\gamma},a_k])a^ia^ja^k.
\end{equation}

We have to compute expressions of type $([[x_{-\theta},b_1],b_2]|[x_{-\gamma},b_3])$ with $b_i\in\g_{d_i}$. Such an expression can be non-zero only if  $d_3-1/2=1-d_1-d_2$ i. e. $d_1+d_2+d_3=3/2$. The possibilities are

\vskip20pt

\centerline{
\begin{tabular}{c|c|c|||c|c|c|||c|c|c}
$d_1$& $d_2$& $d_3$&$d_1$& $d_2$& $d_3$&$d_1$& $d_2$& $d_3$\\
\hline
-1/2&1&1&0&1/2& 1&0& 1&1/2\\
1/2& 0& 1&1/2& 1/2& 1/2&1/2& 1& 0\\
1& -1/2& 1&1& 0& 1/2&1& 1/2& 0\\
1& 1& -1/2\\
\end{tabular}
}
\medskip
For each of the cases above let $S(d_1,d_2,d_3)$ be the corresponding summand in the expression of $G_2(x_{-\theta}\wedge x_{-\gamma})$  given in \eqref{G2}. Then direct computations yield
\begin{align*}
S(-1/2,1,1)&=4\sum_i([[x_{-\theta},\ci],x_{\theta}]|[x_{-\gamma},x_{\theta}])\dci x_{-\theta}^2=0,\\
S(0,1/2,1)&=2\sum_{j}\langle \cj,x_{-\gamma}\rangle\cj xx_{-\theta}=2xx_{-\gamma} x_{-\theta},\\
S(0,1,1/2)&=2\sum_j\langle \cj,x_{-\gamma}\rangle\dduej xx_\theta=2xx_{-\gamma} x_{-\theta},\\
S(1/2,0,1)& =\sum_i[x_i, x_{-\gamma}]\trei x_{-\theta}+x_{-\gamma} xx_{-\theta},\\
S(1/2,1/2,1/2)&=-\tfrac{1}{2}\sum_{j,s}[\cj,x_s] u^j[x_s,x_{-\gamma}] -\tfrac{1}{4}\sum_j\cj u^j x_{-\gamma},\\
S(1/2,1,0)&=\sum_i[x_i, x_{-\gamma}]\trei x_{-\theta}-xx_{-\gamma} x_{-\theta},\notag\\
S(1,-1/2,1)&=-2[x_\theta,x_{-\gamma}]x_{-\theta}^2=-e_{i_0}x_{-\theta}^2,\\
S(1,0,1/2)&=0,\\
S(1,1/2,0)&=\sum_i[x_i, x_{-\gamma}]\trei x_{-\theta}-xx_{-\gamma} x_{-\theta},\notag\\
S(1,1,-1/2)&=4\sum_j\langle x_{-\gamma},\cj\rangle[x_\theta,u^j] x_{-\theta}^2=-4[x_\theta,x_{-\gamma}] x_{-\theta}^2=-2e_{i_0}x_{-\theta}^2.
\end{align*}
Summing up we find
\begin{align*}
G_2(x_{-\theta}\wedge x_{-\gamma})&=3xx_{-\gamma} x_{-\theta} 
 +3\sum_i[x_i, x_{-\gamma}]\trei x_{-\theta} -3e_{i_0}x_{-\theta}^2\\
&-\tfrac{1}{2}\sum_{j,s}[\cj,x_s]u^j[x_s,x_{-\gamma}] -\tfrac{1}{4}\sum_ju_ju^jx_{-\gamma}.
\end{align*}
Note that 
$\sum_ju_ju^j ([x_\theta,u])=\langle u,u\rangle=0,
$
so
\begin{align*}
G_2(x_{-\theta}\wedge x_{-\gamma})&=3xx_{-\gamma} x_{-\theta} 
 +3\sum_i[x_i, x_{-\gamma}]\trei x_{-\theta} -3e_{i_0}x_{-\theta}^2
-\tfrac{1}{2}\sum_{j,r}[\cj,x_r]u^j[x_r,x_{-\gamma}].
\end{align*}
 We now compute
 $g_2(x_{-\theta}\wedge x_{-\gamma})|x_{\theta}\wedge x_ {\gamma})=((Symm(G_2(x_{-\theta}\wedge x_{-\gamma}))(x_{\theta})|x_{\gamma}).$ We start by computing 
 $
\sum_{j,r}( (Symm([\cj,x_r]u^j[x_r,x_{-\gamma}]  )(x_{\theta})|x_{\gamma}).
$
For this we note that, if $u,v,w\in\g_{-1/2}$,
\begin{align*}
([u&,[ v,[w,x_{\theta}]]]|x_\gamma)=([ v,[x_{\theta},w]]]|[u,x_\gamma])=\sum_i\langle v,[w,\trei ]\rangle\langle u,[f_{i_0},\trei]\rangle+\tfrac{1}{2}\langle v,w\rangle\langle u,f_{i_0}\rangle,
\end{align*}
so, letting $\gamma_s$  denote the eigenvalue of $C_{\g^\natural_s}$ on $\g_{-1/2}$,
\begin{align*}
&\sum_{r,j}([[\cj,x_r],[u^j,[[x_r,x_{-\gamma}] ,x_{\theta}]]]|x_{\gamma})\\
&=\sum_{i,j,r}\langle u^j,[[x_r,x_{-\gamma}],\trei ]\rangle\langle [\cj,x_r],[f_{i_0},\trei]\rangle+\tfrac{1}{2}\sum_{j,r}\langle u^j,[x_r,x_{-\gamma}]\rangle\langle [\cj,x_r],f_{i_0}\rangle\\
&=\sum_{i,r}\langle [x_r,[f_{i_0},\trei ]],[[x_r,x_{-\gamma}],\trei ]\rangle+\tfrac{1}{2}\sum_{r}\langle [x_r,f_{i_0}],[x_r,x_{-\gamma}]\rangle\\
&=\sum_{i,r,s}\langle[\trei, [x^s_r,[f_{i_0},\trei ]]],[x^s_r,x_{-\gamma}]\rangle-\tfrac{1}{2}\sum_{r}\langle [x_r,[x_r,f_{i_0}],x_{-\gamma}]\rangle\\
&=\sum_{s}(h^\vee-3/2-\bar h_s^\vee)\sum_r\langle [f_{i_0},x^s_r ],[x^s_r,x_{-\gamma}]\rangle-\tfrac{1}{2}(h^\vee-3/2)\langle f_{i_0},x_{-\gamma}\rangle\\
&=\sum_s(h^\vee-3/2-\bar h_s^\vee)\gamma_s\langle f_{i_0},x_{-\gamma}\rangle-\tfrac{1}{2}(h^\vee-3/2)\langle f_{i_0},x_{-\gamma}\rangle\\
&=\sum_s(h^\vee-2-\bar h_s^\vee)\gamma_s\langle f_{i_0},x_{-\gamma}\rangle=\sum_s(h^\vee-2-\bar h_s^\vee)\gamma_s( x_\gamma|x_{-\gamma}).
\end{align*}
Similarly
\begin{align*}
\sum_{r,j}([[\cj,x_r], [[x_r,x_{-\gamma}],[ u^j,x_{\theta}]]]|x_{\gamma})
&=\sum_s(h^\vee-1-\bar h_s^\vee)\gamma_s( x_\gamma|x_{-\gamma}),\\
\sum_{r,j}([u^j,[ [\cj,x_r],[[x_r,x_{-\gamma} ],x_{\theta}]]]|x_{\gamma})&=\sum_s(h^\vee-2-\bar h_s^\vee)\gamma_s( x_\gamma|x_{-\gamma}),\\
\sum_{r,j}([u^j,[[x_r,x_{-\gamma}],[ [\cj,x_r],x_{\theta}]]]|x_\gamma)&=\sum_s(h^\vee-1-\bar h_1^\vee)\gamma_s( x_\gamma|x_{-\gamma}).
\end{align*}
Finally
\begin{align*}
\sum_{r,j}([[x_r,x_{-\gamma}]&,[ u^j,[[\cj,x_r],x_{\theta}]]]|x_{\gamma})+\sum_{r,j}([[x_r,x_{-\gamma}],[ [\cj,x_r],[u^j ,x_{\theta}]]]|x_{\gamma})\\
&=\sum_{i,j,r}\langle u^j,[[\cj,x_r],\trei ]\rangle\langle  [x_r,x_{-\gamma}],[f_{i_0},\trei]\rangle+\tfrac{1}{2}\sum_{j,r}\langle u^j ,[\cj,x_r]\rangle\langle [x_r,x_{-\gamma}] ,f_{i_0}\rangle\\
&+\sum_{i,j,r}\langle [\cj,x_r],[u^j,\trei ]\rangle\langle  [x_r,x_{-\gamma}],[f_{i_0},\trei]\rangle+\tfrac{1}{2}\sum_{j,r}\langle [\cj,x_r],u^j\rangle\langle [x_r,x_{-\gamma}] ,f_{i_0}\rangle\\
&=2\sum_{i,j,r}\langle [\cj,x_r],[u^j,\trei ]\rangle\langle  [x_r,x_{-\gamma}],[f_{i_0},\trei]\rangle.
\end{align*}
Recall from \eqref{F7} that
\begin{align*}
&\sum_{i,j}\langle [u,x_i],[v,x_j]\rangle\langle[w,x_i],[z,x_j]\rangle-\sum_{i,j}\langle [w,x_i],[v,x_j]\rangle\langle[u,x_i],[z,x_j]\rangle\\
&=
(h^\vee-1)\langle u,w\rangle\langle v,z\rangle+\tfrac{1}{4}\langle u\wedge w,v\wedge z\rangle.
\end{align*}
so
\begin{align*}
&\sum_{r,j}([[x_r,x_{-\gamma}],[ u^j,[[\cj,x_r],x_{\theta}]]]|x_{\gamma})+\sum_{r,j}([[x_r,x_{-\gamma}],[ [\cj,x_r],[u^j ,x_{\theta}]]]|x_{\gamma})\\
&=2\sum_{i,j,r}\langle [\cj,x_r],[u^j,\trei ]\rangle\langle  [x_r,x_{-\gamma}],[f_{i_0},\trei]\rangle\\
&=
-2 (h^\vee-1)\sum_j\langle \cj,x_{-\gamma}\rangle\langle u^j,f_{i_0}\rangle-\tfrac{1}{2}\sum_j\langle \cj\wedge x_{-\gamma},u^j\wedge f_{i_0}\rangle\\
&+2\sum_{i,j,r}\langle [x_{-\gamma},x_r],[u^j,\trei ]\rangle\langle  [x_r,\cj],[f_{i_0},\trei]\rangle\\
&=
-2(h^\vee-1)\langle x_{-\gamma},f_{i_0}\rangle-\tfrac{1}{2}\sum_j\langle \cj,u^j\rangle\langle x_{-\gamma},f_{i_0}]\rangle+\tfrac{1}{2}\sum_j\langle \cj,f_{i_0}\rangle\langle x_{-\gamma},u^j \rangle\\
&+2\sum_{i,r,s}\langle [x_{-\gamma},x^s_r],[\trei , [x^s_r,[f_{i_0},\trei]]]\rangle\\
&=
(2(h^\vee-1)+\tfrac{1}{2}(\dim \g_{1/2}-1))( x_\gamma|x_{-\gamma})+2\sum_s(h^\vee-3/2-\bar h_s^\vee)\sum_{r}\langle[x_{-\gamma},x^s_r],[f_{i_0},x^s_r]\rangle\\
&=
2(h^\vee-1)+\tfrac{1}{2}(2h^\vee-4-1))( x_\gamma|x_{-\gamma})+2\sum_s(h^\vee-3/2-\bar h_s^\vee)\gamma_s( x_\gamma|x_{-\gamma})\\
&=
2\sum_s(h^\vee-\bar h_s^\vee)\gamma_s( x_\gamma|x_{-\gamma}).
\end{align*}
The final outcome is that
\begin{align*}
\sum_{j,r}( (Symm([\cj,x_r]u^j[x_r,x_{-\gamma}]  )(x_{\theta})|x_{\gamma})=
6\sum_s(h^\vee-1-\bar h_s^\vee)\gamma_s( x_\gamma|x_{-\gamma}).
\end{align*}
The other terms are easier:
\begin{align*}
&(Symm(xx_{-\gamma} x_{-\theta})(x_{\theta})|x_{\gamma})=-3/4( x_\gamma|x_{-\gamma}),\\
&(Symm(e_{i_0}x_{-\theta}^2)(x_{\theta})|x_{\gamma})=3( x_\gamma|x_{-\gamma}),\\
&\sum_j(Symm([x_j, x_{-\gamma}]\trej x_{-\theta})(x_{\theta})|x_{\gamma})=-3/2(h^\vee-3/2)( x_\gamma|x_{-\gamma}).
\end{align*}
Summing up
$$(g_2(x_\theta\wedge x_\gamma)|x_{-\theta}\wedge x_{-\gamma})=(-3\sum_s(h^\vee+1/2-\bar h_s^\vee)\gamma_s-45/4)( x_\gamma|x_{-\gamma}).
$$

Using  \eqref{g} we obtain the following result
\begin{proposition}
\begin{equation}\label{finaledelta}
\delta=-\frac{12\sum_s(h^\vee+1/2-\bar h^\vee_s)\gamma_s+45}{144}.
\end{equation}
\end{proposition}
\begin{rem} Relation \eqref{finaledelta} easily gives Drinfeld's formulas for $\d$: if $\kappa$ denotes the Killing form of $\g$, then 
\begin{equation}\label{DF}
\d^\kappa=\begin{cases}
-\frac{1}{32 n^2}\quad&\text{if $\g=sl(n)$,}\\
-\frac{n-4}{16(n-2)^3}\quad&\text{if $\g=so(n)$,}\\
-\frac{n+2}{64(n+1)^3}\quad&\text{if $\g=sp(2n)$,}\\
-\frac{5}{144(\dim\g+2)}\quad&\text{if $\g$ is of exceptional type.}\\
\end{cases}
\end{equation}
\end{rem}
\begin{rem} It is possible to give an alternative formula for $\d$ without using the minimal gradation.  We already observed that we are reduced to evaluate \eqref{enough}. We'll do it by choosing $a=c=u, b=d=v$ with $u,v\in\h$. Let $\{x_\a\}$ be a set of root vectors with 
  $[x_\a,x_\beta]=N_{\a,\beta} x_{\a+\beta}$ for $\a,\beta\in\D, \a\ne -\beta$ and  $N_{\a,\be}=-N_{-\a,-\be}$. Then 
  \begin{equation}\label{altraformula}
(g_2(u\wedge v),u\wedge v)=6\sum_{\a\in\Dp,\be\in\D}\a(u)(\a+\be)(v)(\a(v)\be(u)-\a(u)\be(v))N^2_{\a,\be}.\end{equation}
The computation of \eqref{DF} starting from \eqref{altraformula} is possible but quite less handy than using \eqref{finaledelta}.

\end{rem}

\section{Minimal \texorpdfstring{$W$}{W}-algebras}
It is known by \cite{KW} that for $k\ne -h^\vee$ there is a vertex algebra $W$ strongly and freely generated by fields $L$, $J^v$ with  $v\in \g^\natural$, $G^u$ with $u\in \g_{-1/2}$
with the following  $\lambda$-brackets: $L$ is a Virasoro element  with central charge $\frac{k\,\dim\g}{k+h^\vee}-6k+h^\vee-4$, $J^u$ are primary of conformal weight $\Delta=1$, $G^{v}$ are primary of conformal weight $\Delta=\frac{3}{2}$ and
\begin{enumerate}
\item $[J^v_\l J^w]=J^{[v,w]}+\l\d_{ij}(k+\frac{h^\vee-\bar h^\vee_i}{2})(v|w)$ for $v\in \g^\natural_i$, $w\in \g^\natural_j$;
\item $[J^v_\l G^u]=G^{[v,u]}$ for  $ u\in \g_{-1/2}$, $v\in \g^\natural$;
\item $[G^{u}_\l G^{v}]=A(u,v,k)+\l B(u,v,k)+\frac{\l^2}{2}C(u,v,k)$ for  $u,v\in \g_{-1/2}$
with $C(u,v,k)\in \C$, and  conformal weights of $\Delta(B(u,v,k))=1$ and  $\Delta(A(u,v,k))=2$.
\end{enumerate}

To simplify notation, we will not record the dependence on $k$ in the functions $A,B,C$. We choose the basis $\{x_i\}$ to be the union of orthonormal bases of $\g_r^\natural$. Let $T=L_{-1}$ be the translation operator of $W$.

If $p=a\otimes b\in \g^\natural \otimes \g^\natural$ write $:p:=:J^aJ^b:$. We extend $:\cdot :$ linearly to obtain a map $:\cdot :$ from $\g^\natural \otimes \g^\natural$ to $W$.
Consider $S^2(\g^\natural)=\{a\otimes b+b\otimes a\mid a,b\in\g^\natural\}\subset\g^\natural \otimes \g^\natural$.
Since $:J^uJ^v:=:J^vJ^u:+TJ^{[u,v]}$ and since the elements  $J^{x_i}$, $G^{u_i}$,  $L$ strongly and freely generate $W$, we see that there exist maps $$P:\g_{-1/2}\times \g_{-1/2}\to S^2(\g^\natural),\ \ K,H:\g_{-1/2}\times \g_{-1/2}\to \g^\natural, \ \ Q:\g_{-1/2}\times \g_{-1/2}\to \C$$ such that $A(u,v)$ can be uniquely written as 
$$
A(u,v)=:P(u,v):+TJ^{K(u,v)}+Q(u,v)L.
$$
and 
$$
B(u,v)=J^{H(u,v)}.
$$
By skewsymmetry
$[G^u_\l G^v]=-[G^v_{-\l-T}G^u]$ so
\begin{enumerate}
\item $C(u,v)=-C(v,u)$;
\item $H(u,v)=H(v,u)$;
\item $P(u,v)=-P(v,u)$, $K(u,v)=-K(v,u)+H(v,u)$, $Q(u,v)=-Q(v,u)$,
\end{enumerate}
hence  $C(\cdot,\cdot)$ and $Q(\cdot,\cdot)$ are symplectic forms on $\g_{-1/2}$, and 
$$
P:\bigwedge^2 \g_{-1/2}\to S^2(\g^\natural),\quad H:S^2(\g_{-1/2})\to \g^\natural.
$$
By applying the axioms of vertex algebra, (\S\ 1.5 of \cite{DK})
we find for $a\in\g^\natural$, $v,w\in\g_{-1/2}$:
$$
C([a,v],w)=-C(v,[a,w]).
$$
Since $\g_{1/2}$, as a $\g^\natural$--module, is either irreducible or a sum $U\oplus U^*$ with $U$ irreducible and $U$ inequivalent to $U^*$, we see that, up to a constant, there is a unique symplectic $\g^\natural$-invariant nondegenerate bilinear form on $\g_{-1/2}$. Since $$\langle u,v\rangle:=(x_\theta|[u,v])$$ is such a form, we have that
\begin{equation}\label{Cc}C(\cdot,\cdot)=\Gamma(k)\langle\cdot,\cdot\rangle
\end{equation}
for some constant $\Gamma(k)$.

For $b\in\g^\natural$, let $b^\natural_i$ denote the orthogonal projection of $b$ onto $\g^\natural_i$. Write
 $$
 P(v,w)=\sum_{i,j,r,s}k_{i,j}^{r,s}(v,w)(x_i^r\otimes x_j^s)
 $$
with $k_{i,j}^{r,s}=k_{j,i}^{s,r}$.

\subsection{Jacobi identities between two \texorpdfstring{$G$}{G} and one \texorpdfstring{$J$}{J}} \ 

\noindent By Jacobi identity $[J^x_\l [G^v_\mu G^w]]-[G^v_{\mu}[J^a_\l G^w]]=[[J^a_\l G^v]_{\l+\mu}G^w] $. Explicitly
\begin{align*}
&[J^a_\l :P(v,w):]+[J^a_\l TJ^{K(v,w)}]+\l Q(v,w)J^a+\mu [J^a_\l J^{H(v,w)}]\\
&-:P(v,[a,w]):-TJ^{K(v,[a,w])}-Q(v,[a,w])L-\mu J^{H(v,[a,w])}-\tfrac{\mu^2}{2}c(v,[a,w])\\
&=:P([a,v],w):+TJ^{K([a,v],w)}+Q([a,v],w)L+(\l+\mu)J^{H([a,v],w)}+\tfrac{(\l+\mu)^2}{2}c([a,v],w).
\end{align*}

Using Wick formula ((1.37) \cite{DK})  and sesquilinearity we compute explicitly $[J^a_\l :P(v,w):]$ and $[J^a_\l TJ^{K(v,w)}]$.  Then, equating  the coefficients in $\l,\mu$, we find 
\begin{prop}
\begin{align}
&H(v,w)=\sum_r\frac{2\Gamma(k)}{2k+h^\vee-\bar h_r^\vee}[[x_\theta,v],w]^\natural_r\label{relh}.\\
&K(v,w)=\tfrac{1}{2}H(v,w).\\
&Q([a,v],w)=-Q(v,[a,w]).\\
&P([a,v],w)=-P(v,[a,w])+ad(a)P(v,w).\\
&H([a,v],w)=Q(v,w)a+[a,K(v,w)]\label{kijrs}\\
&+\sum_{i, j,r, s}k_{i,j}^{r,s}(v,w)(2k+h^\vee-\bar h_r^\vee)(a|x_i^r)x_j^s+\sum_{i, j,r, s}k_{i,j}^{r,s}(v,w) [[a,x_i^r],x_j^s].\notag\\
&\sum_{j,r}k^{r,r}_{j,j}(v,w)(2k+h^\vee-\bar h_r^\vee)+Q(v,w)(\frac{k\,\dim\g}{k+h^\vee}-6k+h^\vee-4)=3\Gamma(k)\langle v,w\rangle\label{JacobiL}.
\end{align}

In particular, $Q$ is an invariant symplectic form on $\g_{-1/2}$, hence
\begin{equation}\label{D} Q(u,v)=D(k)\langle u,v\rangle.\end{equation}
\end{prop}

\subsection{Jacobi identities between three \texorpdfstring{$G$}{G}} \ 

\noindent
We need auxiliary formulas. By Wick formula \cite[(1.37)]{DK}
\begin{align*}
&[G^u_\l (:J^yJ^z:)]=:G^{[u,y]}J^z:+:J^yG^{[u,z]}:+\l G^{[[u,y],z]}\\
&=:G^{[u,y]}J^z:+:G^{[u,z]}J^y:+TG^{[y,[u,z]]}+\l G^{[[u,y],z]},\end{align*}
moreover, by sesquilinearity,
\begin{align*}
&[G^u_\l TJ^a]=T[G^u_\l J^a]+\l [G^u_\l J^a]=TG^{[u,a]}+\l G^{[u,a]}.
\end{align*}
By Jacobi identity $[G^u_\l [G^v_\mu G^w]]-[G^v_{\mu}[G^u_\l G^w]]=[[G^u_\l G^v]_{\l+\mu}G^w] $. We compute each term. 

\begin{align*}
&[[G^u_\l G^v]_{\l+\mu}G^w]=[(:P(u,v):+\tfrac{1}{2}TJ^{H(u,v)}+Q(u,v)L+\l J^{H(u,v)}+\tfrac{\l^2}{2}c(u,v))_{\l+\mu}G^w]=\\
&-[G^w_{-\l-\mu-T}:P(u,v):]-\tfrac{1}{2}(\l+\mu)G^{[H(u,v),w]}+Q(u,v)(TG^w+\tfrac{3}{2}(\l+\mu)G^w)+\l G^{[H(u,v),w]}\\
&=-[G^w_{-\l-\mu-T}:P(u,v):]+\tfrac{1}{2}(\l-\mu)G^{[H(u,v),w]}+Q(u,v)(TG^w+\tfrac{3}{2}(\l+\mu)G^w)
\end{align*}
\begin{align*}
[G^u_\l [G^v_\mu G^w]]&=[G^u_\l(:P(v,w):+\tfrac{1}{2}TJ^{H(v,w)}+Q(v,w)L+\mu J^{H(v,w)}+\tfrac{\mu^2}{2}c(v,w))]\\
&=[G^u_\l:P(v,w):]+\tfrac{1}{2}TG^{[u,H(v,w)]}+\tfrac{1}{2}\l G^{[u,H(v,w)]}\\&+\tfrac{1}{2}Q(v,w)TG^u+\tfrac{3}{2}Q(v,w)\l G^u+\mu G^{[u,H(v,w)]} 
\end{align*}
so
\begin{align*}
[G^v_\mu [G^u_\l G^w]]&=[G^v_\mu:P(u,w):]\\
&+\tfrac{1}{2}TG^{[v,H(u,w)]}+\tfrac{1}{2}\mu G^{[v,H(u,w)]}+\tfrac{1}{2}Q(u,w)TG^v+\tfrac{3}{2}Q(u,w)\mu G^v+\l G^{[v,H(u,w)]} .
\end{align*}

Equating the coefficients of $\l, \mu$ and  the constant term we find
\begin{prop}
\begin{align}
&\tfrac{1}{2}G^{[H(u,v),w]}+\tfrac{3}{2}Q(u,v)G^w+\sum_{i, j, r,s}k_{i,j}^{r,s}(u,v)G^{[[w,x_i^r],x_j^s]}\label{1G}\\
&=\tfrac{1}{2}G^{[u,H(v,w)]}+\tfrac{3}{2}Q(v,w)G^u+\sum_{i, j, r,s}k_{i,j}^{r,s}(v,w)G^{[[u,x_i^r],x_j^s]}-G^{[v,H(u,w)]}\notag
\\
&-\tfrac{1}{2}G^{[H(u,v),w]}+\tfrac{3}{2}Q(u,v)G^w+\sum_{i, j, r, s}k_{i,j}^{r,s}(u,v)G^{[[w,x_i^r],x_j^s]}\label{2G}\\
&=G^{[u,H(v,w)]}-\tfrac{1}{2}G^{[v,H(u,w)]}-\tfrac{3}{2}Q(u,w)G^v-\sum_{i, j, r, s}k_{i,j}^{r,s}(u,w)G^{[[v,x_i^r],x_j^s]}\notag
\\
&Q(u,v)TG^w+ 2 \sum_{i, j} k_{i,j}(u,v)TG^{[[w,x_i],x_j]}=\label{3G}\\ 
&\tfrac{1}{2}TG^{[u,H(v,w)]}+\tfrac{1}{2}Q(v,w)TG^u- \sum_{i, j, r,s} k_{i,j}^{r,s}(v,w)TG^{[[u,x^r_i],x^s_j]}\notag\\
&-\tfrac{1}{2}TG^{[v,H(u,w)]}-\tfrac{1}{2}Q(u,w)TG^v+ \sum_{i, j, r, s} k_{i,j}^{r,s}(u,w)TG^{[[v,x^r_i],x^s_j]}\notag
\end{align}
\begin{align}
&-\sum_{i, j,r,s}k_{i,j}^{r,s}(u,v):G^{[w,x_i]}J^{x_j}:\label{4G}\\
&=\sum_{i,j,r,s}k_{i,j}^{r,s}(v,w):G^{[u,x_i]}J^{x_j}:-\sum_{i,j,r,s}k_{i,j}^{r,s}(u,w):G^{[v,x_i]}J^{x_j}:.\notag
\end{align}
\end{prop}

\vskip10pt
Recall from Section \ref{mingrad}
that
$\g_{-1/2}\wedge \g_{-1/2}=\C\oplus V'$, where $V'$ is the $\g^\natural$--module generated by bivectors $u\wedge v$ with $[u,v]=0$, and that $V'$ decomposes with multiplicity one  and no component is trivial. By Proposition \ref{Phinontrivial}, $P_{|V_{\g^\natural}(\l_h)}=f_h(k)\Phi_{|V_{\g^\natural}(\l_h)}$ for some costant $f_h(k)\in \C$. Thus, if $u\wedge v\in V_{\g^\natural}(\l_h)$, 
$$
k^{r,s}_{i,j}(u,v)=f_h(k)(\langle[[u,x^r_i],[v,x^s_j]]\rangle+\langle[[u,x^s_j],[v,x^r_i]]\rangle).
$$
If  $[v,w]=0$, \eqref{kijrs} becomes
\begin{align*}H([a,v],w)&=[a,\sum_r\frac{\Gamma(k)}{2k+h^\vee-\bar h_r^\vee}[[x_\theta,v],w]^\natural_r]
\\&+\sum_{i, j,r, s}k_{i,j}^{r,s}(v,w)(2k+h^\vee-\bar h_r^\vee)(a|x_i^r)x_j^s+\sum_{i, j,r, s}k_{i,j}^{r,s}(v,w) [[a,x_i^r],x_j^s].\end{align*}
Now compute
\begin{align*}
&(H([x^r_i,v],w)|x^s_j)\\
&=([x_i^r,\tfrac{1}{2}H(v,w)]|x^s_j)+\sum_{n,m,r',s'}k_{n,m}^{r',s'}(v,w)(2k+h^\vee-\bar h_{r'}^\vee)(x_i^r|x_n^{r'})(x_m^{s'}|x_j^s)\\&
+\sum_{n,m,r',s'}k^{r',s'}_{n,m}(v,w) ([[x_i^r,x_n^{r'}],x_m^{s'}]|x^s_j)\\
&=([x_i^r,\tfrac{1}{2}H(v,w)]|x^s_j)+k_{i,j}^{r,s}(v,w)(2k+h^\vee-\bar h_r^\vee)+\d_{r,s}\sum_{n,m}k^{r,s}_{n,m}(v,w) ([[x_i^r,x_n^{r}],x_m^{s}]|x^s_j).\\
\end{align*}
Since $[a,H(v,w)]=H([a,v],w)+H(v,[a,w])$ we can rewrite the above relations as
\begin{align*}
&(H([x^r_i,v],w)|x^s_j)-(H(v,[x^r_i,w])|x^s_j)\\
&=2f_h(k)(\langle[[v,x^r_i],[w,x^s_j]]\rangle+\langle[[v,x^s_j],[w,x^r_i]]\rangle)(2k+h^\vee-\bar h_r^\vee)\\
&+\d_{r,s}\sum_{n, m}2f_h(k)(\langle[v,x^r_n],[w,x^s_m]\rangle+\langle[v,x^s_m],[w,x^s_n]\rangle)([[x^r_i,x^r_n],x^s_m]|x^s_j).\end{align*}
More precisely
\begin{align}
&-\frac{\Gamma(k)}{2k+h^\vee-\bar h_s^\vee}(\langle[v,x^r_i],[w,x^s_j]\rangle+\langle[v,x^s_j],[w,x^r_i]\rangle)	\notag\\
&=f_h(k)(\langle[[v,x^r_i],[w,x^s_j]]\rangle+\langle[[v,x^s_j],[w,x^r_i]]\rangle)(2k+h^\vee-\bar h_r^\vee)\label{inter}\\
&+\d_{r,s}\sum_{n, m}f_h(k)(\langle[v,x^r_n],[w,x^s_m]\rangle+\langle[v,x^s_m],[w,x^s_n]\rangle)([[x^r_i,x^r_n],x^s_m]|x^s_j).\notag
\end{align}
If there are at least two simple components in $\g^\natural$, we have
\begin{equation}\label{fh}
f_h(k)=-\frac{\Gamma(k)}
{(2k+h^\vee-\bar h^\vee_s)(2k+h^\vee-\bar h_r^\vee)},
\end{equation}
which is in particular independent from  $h$. 
So we are reduced to determine $f_h(k)$ when $\g^\natural$ is simple or one-dimensional. Recall from the explicit description of the decomposition of $\g_{-1/2}\wedge \g_{-1/2}$ given in Section \ref{mingrad}, that in this case $V'$ is simple. We can therefore drop the superscript from $x_i^r, k_{i,j}^{r,s}$ and we denote $f_h$ simply by $f$. 
Choosing in \eqref{inter} $v,w,i,j$ such that $\langle[v,x_i],[w,x_j]\rangle+\langle[v,x_j],[w,x_i]\rangle\ne 0$ we can write 
$$
f(k)=\frac{\Gamma(k)}{(\xi k+\eta)(2k+h^\vee-\bar h_1^\vee)}
$$
with $\xi\ne0$.
In particular
\begin{align*}
&-(\xi k+\eta)(\langle[v,x_i],[w,x_j]\rangle+\langle[v,x_j],[w,x_i]\rangle)\\
&=(\langle[v,x_i],[w,x_j]\rangle+\langle[v,x_j],[w,x_i]\rangle)(2k+h^\vee-\bar h_1^\vee)\\
&+\sum_{n, m}(\langle[v,x_n],[w,x_m]\rangle+\langle[v,x_m],[w,x_n]\rangle)([[x_i,x_n],x_m]|x_j),
\end{align*}
so $\xi=-2$ and 
\begin{align*}
&-(\eta+(h^\vee-\bar h_1^\vee))(\langle[v,x_i],[w,x_j]\rangle+\langle[v,x_j],[w,x_i]\rangle)\\
&=\sum_{n,m}(\langle[v,x_n],[w,x_m]\rangle+\langle[v,x_m],[w,x_n]\rangle)([[x_i,x_n],x_m]|x_j).
\end{align*}
To compute $\eta$ we first observe that
\begin{align*}
&\sum_{n,m}(\langle[v,x_n],[w,x_m]\rangle+\langle[v,x_m],[w,x_n]\rangle)([[x_i,x_n],x_m]|x_j)\\
&=-\sum_{n,m}(\langle[v,x_n],[w,x_m]\rangle+\langle[v,x_m],[w,x_n]\rangle)([[x_i,x_n],x_j]|x_m)\\
&=-\sum_{n}(\langle[v,x_n],[w,[[x_i,x_n],x_j]]\rangle+\langle[v,[[x_i,x_n],x_j]],[w,x_n]\rangle),\\
\end{align*}
which implies, for any $a,b\in \g^\natural$
\begin{align*}
&(\eta+(h^\vee-\bar h_1^\vee))(\langle[v,a],[w,b]\rangle+\langle[v,b],[w,a]\rangle)\\
&=\sum_{n}(\langle[v,x_n],[w,[[a,x_n],b]]\rangle+\langle[v,[[a,x_n],y]],[w,x_n]\rangle).\\
\end{align*}

Next we need some formulas: let $C_{\g^\natural}=\sum_i x_i^2$ be the Casimir element of $\g^\natural$ and $C_{\g_0}$ the Casimir element of $\g_0$. Since $C_{\g^\natural}=\C_{\g_0}-2x^2$, by Lemma 5.1 of \cite{KW} we have that $\sum_i[x_i,[x_i,v]]=(h^\vee-\tfrac{3}{2})v$.
Now a lengthy computation yields
\begin{align*}
&\sum_{n}(x_\theta|[[v,x_n],[w,[[a,x_n],b]]]+[[v,[[a,x_n],b]],[w,x_n]])\\
&=(h^\vee-3/2-2\bar h_1^\vee)(\langle[v,b],[w,a]\rangle+\langle[v,a],[w,b]\rangle)\\
&+\sum_{n}(\langle[b,[v,x_n]],[a,[w,x_n]]\rangle+\langle[a,[v,x_n]],[b,[w,x_n]]\rangle).
\end{align*}

Consider the map $\Psi:\wedge^2\g_{-1/2}\to S^2(\g^\natural)^*$ defined by polarizing \eqref{Phi}:
$$
\Psi(v\wedge w)(a,b)=\langle[v,a],[w,b]\rangle+\langle[v,b],[w,a]\rangle,
$$
and note that 
$$
\sum_{n}(\langle[b,[v,x_n]],[a,[w,x_n]]\rangle+\langle[a,[v,x_n]],[b,[w,x_n]]\rangle)=\Psi(\sum_n [v,x_n]\wedge [w,x_n])(a,b).
$$
By Lemma \ref{lemma}, \eqref{F8}, we have
$$
\sum_{n}(\langle[b,[v,x_n]],[a,[w,x_n]]\rangle+\langle[a,[v,x_n]],[b,[w,x_n]]\rangle)=-\tfrac{1}{2}(\langle[b,v],[a,w]\rangle+\langle[a,v],[b,w]\rangle).
$$
The outcome is that
\begin{align*}
&\sum_{n}(\langle[v,x_n],[w,[[a,x_n],b]]\rangle+\langle[v,[[a,x_n],b]],[w,x_n]\rangle)\\
&=(h^\vee-2-2\bar h_1^\vee)(\langle[v,b],[w,a]\rangle+\langle[v,a],[w,b]\rangle).
\end{align*}
Thus $\eta=-2-\bar h_1^\vee$. In particular
\begin{equation}\label{comb}f(k)=-\frac{\Gamma(k)}{4(k+\frac{h^\vee-\bar h_1^\vee}{2})(k+\frac{\bar h^\vee_1}{2}+1)}.\end{equation}
This ends  the computation of the proportionality factor $f_h(k)$.
\vskip5pt
It remains to compute $Q(v,w)$ and $P(v,w)$ with $[v,w]\ne0$. 
To this end introduce
 $$TR_r(v,w):=\sum_{j}k^{r,r}_{j,j}(v,w).$$ 
Relation \eqref{kijrs} gives

\begin{align*}
(H([x^r_i,v],w)|x^r_i)&=Q(v,w)+k^{r,r}_{i,i}(v,w)(2k+h^\vee-\bar h_r^\vee)\\
&+\sum_{n,m}k^{r,r}_{n,m}(v,w) (([[x^r_i,x^r_n],x^r_m]|x^r_i).\notag
\end{align*}
So
\begin{align*}
\sum_i(H([x^r_i,v],w)|x^r_i)&=\dim\g^\natural_r Q(v,w)+TR_r(v,w)(2k+h^\vee-\bar h_r^\vee)\\
&+\sum_i\sum_{n, m}k^{r,r}_{n,m}(v,w) ([[x^r_i,x^r_n],x^r_m]|x^r_i).\notag
\end{align*}
Using the relation $\sum_i[x^r_i,[x^r_i,a]]=2(\bar h_r^\vee)a$ for $a\in\g^\natural_r$ we obtain
\begin{align*}
\sum_i(H([x^r_i,v],w)|x^r_i)&=\dim\g_r^\natural Q(v,w)+ TR_r(v,w)(2k+h^\vee+\bar h_r^\vee).
\end{align*}
Since
$$
\sum_i(H([x^r_i,v],w)|x^r_i)=\tfrac{2\Gamma(k)\gamma_r}{2k+h^\vee-\bar h_r^\vee}\langle v,w\rangle,
$$
 we have, recalling that $Q(v,w)=D(k)\langle v,w\rangle$,
\begin{equation}\label{trr}
TR_r(v,w)=\frac{1}{2k+h^\vee+\bar h_r^\vee}\left(-\dim \g^\natural_r D(k)+\tfrac{2\Gamma(k)\gamma_r}{2k+h^\vee-\bar h_r^\vee}\right)\langle v,w\rangle.
\end{equation}
Relation \eqref{JacobiL} becomes

$$
\sum_{r}\frac{2k+h^\vee-\bar h_r^\vee}{2k+h^\vee+\bar h_r^\vee}\left(-\dim \g^\natural_r D(k)+\tfrac{2\Gamma(k)\gamma_r}{2k+h^\vee-\bar h_r^\vee}\right)+D(k)(\frac{k\,\dim\g}{k+h^\vee}-6k+h^\vee-4)=3\Gamma(k).
$$
Solving for $D(k)$ we find that 
$$D(k)=\Gamma(k) E(k),$$
where $E(k)$ is a complicated but explicit  rational function  in $k,h^\vee,\bar h_r^\vee,\gamma_r, \dim\g,\dim\g_r^\natural $. In turn, substituting in \eqref{trr}, $TR_r$ can be expressed as 
\begin{equation}\label{trrr}
TR_r(v,w)=\Gamma(k)  E_r(k)\langle v,w\rangle.
\end{equation}

Let $\{u_i\}$, $\{u^i\}$ be dual bases of $\g_{-1/2}$: $\langle u_i,u^j\rangle=\d_{i,j}$. We have to compute $P(\sum_i u_i\wedge u^i).$ If $u,v\in\g_{-1/2}$, then
$$
u\wedge v=\frac{\langle u,v\rangle}{\dim \g_{1/2}}\sum_i u_i\wedge u^i+s, \quad s\in V_2.
$$
By covariance of $P$, we have $P(\sum_i u_i\wedge u^i)= \sum_{i,r} \alpha_rx_i^r\otimes x_i^r$, thus
$$
\sum_{i,r,s,m,n}k^{r,s}_{m,n}(u_i,u^i)x^r_m\otimes x^s_n= \sum_{i,r}\alpha_r x_i^r\otimes x^r_i,
$$
hence $\sum_{i}k^{r,s}_{m,n}(u_i,u^i)=\d_{r,s}\d_{m,n}\a_r$.
In particular
$
\dim \g^\natural_r\alpha_r=\sum_{i,j}k_{j,j}^{r,r}(u_i,u^i)
$
hence 
$$
\dim \g^\natural_r\a_r=\sum_i TR_r(u_i,u^i)=\Gamma(k)  E_r(k)\sum_i\langle u_i,u^i\rangle=\Gamma(k)  E_r(k)\dim \g_{-1/2}.
$$
Since $\Phi$ is equivariant we have likewise
$$
\sum_{i,r,s,m,n}(\langle [u_i,x^r_m],[u^i,x^s_n]\rangle+\langle [u_i,x^s_m],[u^i,x^r_m]\rangle)x^r_m\otimes x^s_n= \sum_{i,r} \beta_r x^r_i\otimes x^r_i,
$$
hence
$$
\sum_{i}(\langle [u_i,x^r_m],[u^i,x^s_n]\rangle+\langle [u_i,x_n^s],[u^i,x^r_m]\rangle)=\d_{r,s}\d_{m,n}\be_r
$$
and
\begin{align*}
\dim \g^\natural_r\beta_r&=\sum_{i,m}(\langle [u_i,x^r_m],[u^i,x^r_m]\rangle+\langle [u_i,x_m^r],[u^i,x_m^r]\rangle)=-2\sum_{i,m}\langle[x_m^r, [x_m^r,u_i]],u^i\rangle\\
&=-2\gamma_r\sum_i\langle u_i,u^i\rangle=-2\gamma_r\dim \g_{-1/2}.
\end{align*}
Since $P(s)=f(k)\Phi(s)$,
\begin{align*}
&P(s)
=\sum_{n,m,r,s}f(k)(\langle [u,x^r_m],[v,x^s_n]\rangle+\langle[u,x^s_n],[v,x^r_m]\rangle)(x^r_m\otimes x^s_n)\\
&-\frac{\langle u,v\rangle}{\dim \g_{1/2}}f(k)\sum_{i, m,n,r,s}(\langle[u_i,x^r_m],[u^i,x^s_n]\rangle +\langle[u_i,x^s_n],[u^i,x^r_m]\rangle)(x^r_m\otimes x^s_n)=\\
&\sum_{ m,n,r,s}f(k)(\langle [u,x^r_m],[v,x^s_n]\rangle+\langle[u,x^s_n],[v,x^r_m]\rangle)(x^r_m\otimes x^s_n)-\frac{\langle u,v\rangle}{\dim \g_{1/2}}f(k)\sum_{i,r}\be_r(x^r_i\otimes x^r_i)=\\
&\sum_{m,n,r,s}f(k)(\langle [u,x^r_m],[v,x^s_n]\rangle+\langle[u,x^s_m],[v,x^r_m]\rangle)(x^r_m\otimes x^s_n)+\langle u,v\rangle f(k)\sum_{m,r}\frac{2\gamma_r }{\dim \g^\natural_r}(x_m^r\otimes x_m^r).
\end{align*}
The outcome is that
\begin{align}\label{ar}
P(u\wedge v)&=\langle u,v\rangle\sum_{r,m}\frac{\Gamma(k)  E_r(k)+2\gamma_rf(k)}{\dim \g^\natural_r}(x_m^r\otimes x_m^r)\\
&+\sum_{m,n,r,s}f(k)(\langle [u,x^r_m],[v,x^s_n]\rangle+\langle[u,x^s_n],[v,x^r_m]\rangle)(x^r_m\otimes x^s_n).\notag
\end{align}

Observe from \eqref{fh} that  $f$ does not depend on the choice of $r,s$, hence $\{\bar h_r^\vee\}$ has at most two elements, 
and if there are more than two components,  $\{\bar h_r^\vee\}$ is a singleton. 

We now deal with the case in which $\g^\natural$ has three components. Suppose that  $\g^\natural$ has a nontrivial center.  Then $\bar h_1^\vee=0$, also $\bar h_2^\vee, \bar h_3^\vee$ vanish and $\g^\natural$ is $3$-dimensional abelian. This is not possible, hence $\g^\natural$ is semisimple,   
 $\a_{i_0}$ is long and is a  node of degree $3$ in the Dynkin diagram of $\g$. Therefore one of the components, say $\g_1^\natural$,  has to be $sl(2)$. In particular $\bar h_1^\vee=h^\vee_1=2$. By the above remark, we have $\bar h_2^\vee=\bar h^\vee_3=2$ and indeed $\nu_2=\nu_3=1$. Hence all components are isomorphic to $sl(2)$ and this forces $\g$ to be of type $D_4$. 
 
Set
$$
p(k)=\begin{cases}
(k+\tfrac{h^\vee-\bar h_1^\vee}{2})(k+\tfrac{h^\vee-\bar h_2^\vee}{2})&\text{if $\g^\natural$ has two components},\\[5pt]
(k+\tfrac{h^\vee-\bar h_1^\vee}{2})(k+\tfrac{\bar h_1^\vee}{2}+1)&\text{otherwise.}
\end{cases}
$$
Observe that, combining  \eqref{fh}  and \eqref{comb}, we have  $f(k)=-\frac{\Gamma(k)}{4p(k)}$ in all cases.
We summarize our findings in the following proposition.
\begin{prop}
There are explicitly computable rational functions $a_r(k)$, $b(k)$, $c(k)$, $d_r(k)$ such that, up to  a constant $C$,
\begin{align*}
&[{G^v}_\l G^w]=\\&C\left(\langle v,w\rangle\sum_{i,r}a_r(k):J^{x^r_i}J^{x^r_i}:+b(k)\sum_{i, j,r}(\langle[v,x^r_i],[w,x^r_j]\rangle+\langle[v,x^r_j],[w,x^r_i]\rangle):J^{x^r_i}J^{x^r_j}:\right)+\\
&C\left(c(k)\langle v,w\rangle L+\sum_rd_r(k)\left(\tfrac{1}{2}TJ^{[[x_\theta,v],w]_r^\natural}+\l J^{[[x_\theta,v],w]_r^\natural}\right)+\tfrac{\l^2}{2}\langle v,w\rangle\right).
\end{align*}
More precisely
$$
b(k)=-\frac{1}{4p(k)},\quad d_r(k)=\frac{1}{k+\tfrac{h^\vee-\bar h_r^\vee}{2}},
$$
while $a_r(k)$ and $c(k)$ are certain rational functions of degree respectively $-2$ and $-1$.
\end{prop}

If we set  $\varphi(u,v)(w,z)= \sum_{i, j,r,s} k^{r,s}_{i,j}(u,v)\langle[[w,x^r_i],x^s_j],z\rangle$ then $\varphi$ is alternating in $w,z$ so it defines a map $\varphi:\bigwedge^2\g_{-1/2}\to \bigwedge^2\g_{1/2}$ and it is $\g^\natural$-equivariant, since
$\varphi= \pi\circ Symm \circ \Phi$  where $\pi$ is the action of $\g^\natural$ on $\g_{-1/2}$.

Relations \eqref{1G}, \eqref{2G}, \eqref{3G} then become
\begin{align}
&\tfrac{1}{2}\langle[H(u,v),w],z\rangle+\tfrac{3}{2}Q(u,v)\langle w,z\rangle+\varphi(u,v)(w,z)\label{GG1}\\
&=\tfrac{1}{2}\langle[u,H(v,w)],z\rangle+\tfrac{3}{2}Q(v,w)\langle u,z\rangle+\varphi(v,w)(u,z)-\langle[v,H(u,w)],z\rangle\notag,
\\
&-\tfrac{1}{2}\langle[H(u,v),w],z\rangle+\tfrac{3}{2}Q(u,v)\langle w,z\rangle+\varphi(u,v)(w,z)\label{GG2}\\
&=\langle[u,H(v,w)],z\rangle-\tfrac{1}{2}\langle[v,H(u,w)],z\rangle-\tfrac{3}{2}Q(u,w)\langle v,z\rangle-\varphi(u,w)(v,z)\notag,
\\
&Q(u,v)\langle w,z\rangle+ 2\varphi(u,v)(w,z)=\label{GG3}\\ 
&\tfrac{1}{2}\langle[u,H(v,w)],z\rangle+\tfrac{1}{2}Q(v,w)\langle u,z\rangle- \varphi(v,w)(u,z)\notag\\
&-\tfrac{1}{2}\langle[v,H(u,w)],z\rangle-\tfrac{1}{2}Q(u,w)\langle v,z\rangle+ \varphi(u,w)(v,z).\notag
\end{align}

\begin{lemma}
Assume $C\ne0$ for almost all $k$ and set 
$$
R(k)=\begin{cases}
-\frac{\sum_rd_r(k)\Vert(h_{\a_{i_0}})^\natural_r\Vert^2+3c(k) +2\sum_{r}a_r(k)\gamma_r}{2b(k)}&\text{if $\g^\natural$ has two components,}\\[10pt]
-\frac{3/2d_1(k)+3c(k) +(2h^\vee-3)a_1(k)}{2b(k)}&\text{otherwise.}
\end{cases}
$$
Then $R(k)$ does not depend on $k$. More precisely
$$
R(k)=\begin{cases}
\frac{3 - 4 h^\vee +2 \sum_r \bar h_r^\vee\Vert(h_{\a_{i_0}})^\natural_r\Vert^2}{2}&\text{if $\g^\natural$ has two components,}\\[10pt]
\frac{3-4 h^\vee + 3\bar h_1^\vee}{2}&\text{otherwise.}\end{cases}
$$
\end{lemma}
\begin{proof}
Choose $v=u$ in \eqref{GG1}. Then we obtain
\begin{align}\label{HGG1}
\tfrac{1}{2}(\langle[H(u,u),w]+[u,H(u,w)],z\rangle)-\tfrac{3}{2}Q(u,w)\langle u,z\rangle=\varphi(u,w)(u,z).
\end{align}
Using the explicit formulas
$$H(v,w)=C\sum_rd_r(k)[[x_\theta,v],w]^\natural_r,\quad Q(v,w)=Cc(k)\langle v,w\rangle$$
and \eqref{F7} we find
\begin{align*}
&\tfrac{C}{2}\sum_rd_r(k)\sum_i(\langle u,[u,x_i^r]\rangle\langle[x_i^r,w],z\rangle+\langle u,[w,x_i^r]\rangle\langle[u,x_i^r],z\rangle)-\tfrac{3C}{2}c(k)\langle u,w\rangle\langle u,z\rangle=\\&\varphi(u,w)(u,z).
\end{align*}
We first evaluate the right hand side of \eqref{HGG1}: recall that 
$$k_{i,j}^{r,s}(u,v)=Cb(k)(\langle [u,x^r_i] ,[v,x^s_j]\rangle+\langle [u,x^s_j] ,[v,x^r_i]\rangle)+C\d_{i,j}\d_{r,s}a_r(k)\langle u,v\rangle,
$$
so
\begin{align}\label{eqq}
\varphi(u,w)(u,z)&=\sum_{i,j,r,s}Cb(k)(\langle [u,x^r_i] ,[w,x^s_j]\rangle+\langle [u,x^s_j] ,[w,x^r_i]\rangle)\langle[[u,x^r_i],x^s_j],z\rangle\\&+C\sum_{i,r}a_r(k)\langle[[u,x^r_i],x^r_i],z\rangle\langle u,w\rangle\notag\\\notag
&=\sum_{i,j,r,s}Cb(k)(\langle [u,x^r_i] ,[w,x^s_j]\rangle+\langle [u,x^s_j] ,[w,x^r_i]\rangle)\langle[[u,x^r_i],x^s_j],z\rangle\\&+C\sum_{r}a_r(k)\gamma_r\langle u,z\rangle\langle u,w\rangle.\notag
\end{align}

Take $u=f_{i_0}, w=z=[x_{-\theta},e_{i_0}]$. Passing to dual bases $\{x_i^r\},\{x^i_r\}$ of $\g^\natural_r$ and using \eqref{eqq}, relation \eqref{HGG1} becomes
\begin{align}\label{eq}
&\tfrac{C}{2}\sum_{r,i}d_r(k)\left(\langle f_{i_0},[f_{i_0},x_i^r]\rangle\langle[x^i_r,[x_{-\theta},e_{i_0}]],[x_{-\theta},e_{i_0}]\rangle+\langle f_{i_0},[[x_{-\theta},e_{i_0}],x_i^r]\rangle\langle[f_{i_0},x^i_r],[x_{-\theta},e_{i_0}]\rangle\right)\\
&-\tfrac{3C}{2}c(k)\langle f_{i_0},[x_{-\theta},e_{i_0}]\rangle^2=\notag\\\notag
&\sum_{i,j,r,s}Cb(k)\left(\langle [f_{i_0},x^r_i] ,[[x_{-\theta},e_{i_0}],x^s_j]\rangle+\langle [f_{i_0},x^s_j] ,[[x_{-\theta},e_{i_0}],x^r_i]\rangle\right)\langle[[f_{i_0},x^r_i],x^s_j],[x_{-\theta},e_{i_0}]\rangle\notag\\&+C\sum_{r}a_r(k)\gamma_r\langle f_{i_0},[x_{-\theta},e_{i_0}]\rangle^2.\notag
\end{align}

We now evaluate the left hand side of \eqref{eq}.
Assume first that $\g^\natural$ has two components, then $-\theta+2\a_{i_0}$ is not a root.
By weight considerations, $\langle f_{i_0},[f_{i_0},x_i^r]\rangle$ can be non zero only if $x_i^r$ has weight $-\theta+2\a_{i_0}$. 
Arguing in the same way, we also conclude that in the second sum $x_i^r$ should belong to $\h$, so that the left hand side of \eqref{eq} simplifies to
$$
\tfrac{C}{2}(\sum_rd_r(k)\sum_i\langle f_{i_0},[[x_{-\theta},e_{i_0}],h_i^r]\rangle\a_{i_0}(h_i^r)-3c(k) \langle f_{i_0},[x_{-\theta},e_{i_0}]\rangle)\langle f_{i_0},[x_{-\theta},e_{i_0}]\rangle,
$$
where $\{h_i^r\}$ is an orthonormal basis of $\h_r^\natural$. The above formula can be further reduced to 
\begin{equation}\label{eqqqq}
-\tfrac{C}{2}(\sum_rd_r(k)\Vert(h_{\a_{i_0}})^\natural_r\Vert^2+3c(k) ))\langle f_{i_0},[x_{-\theta},e_{i_0}]\rangle^2.
\end{equation}
Assume now that $\g^\natural$ has only one component (i.e., it is simple or $1$-dimensional).
Then, if $u,v\in\g_{-1/2}$,
\begin{align*}[[x_\theta,u],v]&=[[x_\theta,u],v]^\natural+\langle u,v\rangle x.
\end{align*}
In this case  \eqref{HGG1} becomes
$$
\tfrac{C}{2}(d_1(k)(\langle[[[x_\theta,u],u]^\natural,w]+[u,[[x_\theta,u],w]^\natural],z\rangle)-\tfrac{3C}{2}Q(u,w)\langle u,z\rangle=\varphi(u,w)(u,z).
$$
Since
\begin{align*}
[[[x_\theta,u],u]^\natural,w]+[u,[[x_\theta,u],w]^\natural]=[[[x_\theta,u],u],w]+[u,[[x_\theta,u],w]]-\langle u,w\rangle [u,x]
\end{align*}
and $[u,[[x_\theta,u],w]]=-[[[x_\theta,u],u],w]+[[x_\theta,u],[u,w]]$, we see that
\begin{align*}
[[[x_\theta,u],u]^\natural,w]+[u,[[x_\theta,u],w]^\natural]&=[[x_\theta,u],[u,w]]-\tfrac{1}{2}\langle u,w\rangle u\\
&=2\langle u,w\rangle[[x_\theta,u],x_{-\theta}]-\tfrac{1}{2}\langle u,w\rangle u
=-\tfrac{3}{2}\langle u,w\rangle u.
\end{align*}
The upshot is that
$$
-\tfrac{C}{2}(\tfrac{3}{2}d_1(k)+3c(k))\langle u,w\rangle\langle u, z\rangle=\varphi(u,w)(u,z).
$$
Substituting $u=f_{i_0}$ and $w=z=[x_{-\theta},e_{i_0}]$ the left hand side of \eqref{eq} becomes
\begin{equation}\label{524}
-\tfrac{C}{2}(\tfrac{3}{2}d_1(k)+3c(k))\langle f_{i_0},[x_{-\theta},e_{i_0}]\rangle^2.
\end{equation}
Finally, if $\g^\natural$ has three components, 
since $d_1(k)=d_2(k)=d_3(k)$, formula \eqref{eqq} becomes
$$
-\tfrac{C}{2}(d_1(k)\Vert(h_{\a_{i_0}})^\natural\Vert^2+3c(k) ))\langle f_{i_0},[x_{-\theta},e_{i_0}]\rangle^2$$ which is indeed \eqref{524}.
To evaluate the right hand side of \eqref{eqq} in the current case, notice that $\gamma_r$ does not depend on $r$ (see also \eqref{Pan} below). From this and relation \eqref{ar} we deduce  that $a_r(k)$ does not depend on $r$. 
It follows, using \eqref{Pann} below,   that 
$$
C\sum_{r}a_r(k)\gamma_r\langle f_{i_0},[x_{-\theta},e_{i_0}]\rangle^2=\tfrac{C}{2} a_1(k)(2h^\vee -3)\langle f_{i_0},[x_{-\theta},e_{i_0}]\rangle^2,
$$
as in the case when  $\g^\natural$ has only one component.
\vskip5pt
The final outcome is that
\begin{align*}
&\sum_{i,j,r,s}(\langle [f_{i_0},x^r_i] ,[[x_{-\theta},e_{i_0}],x^s_j]\rangle+\langle [f_{i_0},x^s_j] ,[[x_{-\theta},e_{i_0}],x_i]\rangle)\langle[[f_{i_0},x^r_i],x^s_j],[x_{-\theta},e_{i_0}]\rangle\\&=R(k)\langle f_{i_0},[x_{-\theta},e_{i_0}]\rangle^2,
\end{align*}
where
$$
R(k)=\begin{cases}
-\frac{\sum_rd_r(k)\Vert(h_{\a_{i_0}})^\natural_r\Vert^2+3c(k) +2\sum_{r}a_r(k)\gamma_r}{2b(k)}&\text{if $\g^\natural$ has two components,}\\[10pt]
-\frac{3/2d_1(k)+3c(k) +(2h^\vee-3)a_1(k)}{2b(k)}&\text{otherwise.}
\end{cases}
$$
It follows that $R(k)$  does not depend on $k$, hence it equals the value of its limit for $k\to\infty$. This limit is
$$
\lim_{k\to\infty}R(k)=\begin{cases}
\frac{3 - 4 h^\vee +2 \sum_r \bar h^\vee_r \Vert(h_{\a_{i_0}})^\natural_r\Vert^2}{2}&\text{if $\g^\natural$ has two components,}\\[10pt]
\frac{3-4 h^\vee + 3\bar h_1^\vee}{2}&\text{otherwise.}
\end{cases}
$$
\end{proof}

There are several relations among the values $\gamma_i,\dim \g^\natural_i,\dim\g,\bar h^\vee_i$ and $h^\vee$. 
Indeed, if $\g^\natural_i$ is abelian, then $\g^\natural_i=\C\varpi$ with $\varpi$ as in \eqref{varpi}. As noted in Section \ref{mingrad}, $\varpi$ acts on $\g_{-1/2}$ as $\pm I$, so the eigenvalue of $C_{\g^\natural_i}=\tfrac{\varpi^2}{(\varpi|\varpi)}$ is $\tfrac{1}{(\varpi|\varpi)}$. On the other hand $(\varpi|\varpi)=\tfrac{tr(ad(\varpi)^2)}{2h^\vee}=\tfrac{2\dim\g_{-1/2}}{2h^\vee}$, so we conclude that
 $\gamma_i=\frac{\dim\g^\natural_ih^\vee}{2
   (h^\vee-2)}$. Since $h_i^\vee=0$, this formula can be written as 
\begin{equation}\label{Pan}
\gamma_i=\frac{\dim\g^\natural_i(h^\vee-\bar h_i^\vee)}{2
   (h^\vee-2)}.
\end{equation}
By \cite[(2.2)]{P}, the index $ind_{\g_i^\natural}(\g_{1/2}\oplus\g_{-1/2})$
is $(h^\vee-\bar h^\vee_i)/\bar h^\vee_i$. The same index can be computed as $ind_{\g_i^\natural}(\g_{1/2})+ind_{\g_i^\natural}(\g_{-1/2})=2 \,ind_{\g_i^\natural}(\g_{-1/2})$ and this last quantity is computed by \cite[(1.3)]{P} to be $\frac{\gamma_i \dim\g_{-1/2}}{2\bar h_i^\vee \dim\g_i^\natural}$. Since $\dim\g_{-1/2}=2(h^\vee-2)$, it follows that 
\eqref{Pan} holds in these cases too.

By \eqref{casgm}, 
\begin{equation}\label{Pann}
\sum_r\gamma_r=h^\vee-3/2\end{equation} so, 
if $\g^\natural$ has two components one can solve for $\dim\g^\natural_2$ and obtain that
$$
\dim\g^\natural_2=\frac{\dim\g^\natural_1(\bar h_1^\vee-h^\vee)+
   2 (h^\vee)^2-7
   h^\vee+6}{h^\vee-\bar h_2^\vee}.
$$
Moreover, by \eqref{dimgm}, 
\begin{equation}\label{newdim}\dim \g=4h^\vee-5+\dim\g^\natural.\end{equation} 
Our analysis provides more refined  relations.
\begin{proposition}\label{dim} If $\g^\natural$ has  two components, then
\begin{equation}\label{3}
\dim \g= \frac{(h^\vee+1) \left(2 (h^\vee)^2+h^\vee
   (\bar h_1^\vee-2)-\bar h_1^\vee
   (\bar h_1^\vee+2)\right)}{(\bar h_1^\vee+2) (h^\vee-\bar h_1^\vee)},\quad \bar h_1^\vee+\bar h_2^\vee=
   h^\vee-2.
\end{equation}
Otherwise
\begin{equation}\label{1}
\dim\g= \frac{2 \left(5 (h^\vee)^2-h^\vee-6\right)}{h^\vee+6},\quad \bar h_1^\vee=
   \frac{2 (h^\vee-3)}{3}
\end{equation}
or
\begin{equation}\label{2}
\dim \g= 2 (h^\vee)^2-3 h^\vee+1,\quad \bar h_1^\vee=
   h^\vee-1.
\end{equation}
Moreover, \eqref{2} occurs if and only if $\g^\natural$ is simple and $\a_{i_0}$ is short.
\end{proposition}
\begin{proof}
Write explicitly the rational function $R(k)-\lim_{k\to\infty}R(k)
$ as $P/Q$ with $P,Q$ polynomials in $k,h^\vee,\bar h^\vee_i$ and $\dim \g$. Since $P$ is identically zero,  by equating its coefficients to zero and solving the system of equations with respect to $\dim \g$ and $h^\vee_i/\nu_i$ we get the above formulas.

To finish the proof we show that  $\bar h_1^\vee=
   h^\vee-1$ if and only if  $\g^\natural$ is simple  and $\a_{i_0}$ is short.  Let $\Sigma$ be the set of simple roots of $\g^\natural$ and $\theta_\Sigma$ its highest root. Write $\theta=\sum_{\a\in\Pi}m_\a \a, \, \theta_\Sigma=\sum_{\a\in\Sigma}n_\a \a$ and note that  $n_\a\leq m_\a$ for all $\a\in\Sigma$. If $\bar h_1^\vee$, since $\nu_1=2/(\theta_\Sigma|\theta_\Sigma)$, we have 
   $$\bar h_1^\vee=\tfrac{1}{2}((\theta_\Sigma|\theta_\Sigma)+\sum_{\a\in\Sigma}(\a|\a)n_\a=\tfrac{1}{2}\sum_{\a\in\Pi}(\a|\a)m_\a=(\a_{i_0}|\a_{i_0})+\tfrac{1}{2}\sum_{\a\in\Sigma}(\a|\a)m_\a,$$
 hence
$$ (\theta_\Sigma|\theta_\Sigma)-2(\a_{i_0}|\a_{i_0})=\sum_{\a\in\Sigma}(\a|\a)(m_\a-n_\a)\geq 0.$$
It follows that $\a_{i_0}$ is short.  Then $\theta-2\a_{i_0}$ is a root of $\g^\natural$ and this forces $\g^\natural$ to be simple, otherwise the support of $\theta-2\a_{i_0}$ would be disconnected.

Assume now that $\g^\natural$ is simple and $\a_{i_0}$ is short. Then $\theta-2\a_{i_0}$ is a root of $\g^\natural$. This forces $\theta_\Sigma=\theta-2\a_{i_0}$ and $(\theta|\a_{i_0}^\vee)=2$, so that 
$(\a_{i_0}|\a_{i_0})=1$. We have 
\begin{align*}h^\vee-1&=\sum_{\a\in\D} \tfrac{(\a|\a)}{2}m_\a=\sum_{\a\in\Sigma} \tfrac{(\a|\a)}{2}n_\a+(\a_{i_0}|\a_{i_0})
=\tfrac{(\theta_\Sigma|\theta_\Sigma)}{2}\sum_{\a\in\Sigma} \tfrac{(\a|\a)}{(\theta_\Sigma|\theta_\Sigma)}n_\a+(\a_{i_0}|\a_{i_0})\\
&=\bar h_1^\vee-\tfrac{(\theta_\Sigma|\theta_\Sigma)}{2}+(\a_{i_0}|\a_{i_0})=\bar h_1^\vee,
\end{align*}
since $(\theta_\Sigma|\theta_\Sigma)=(\theta-2\a_{i_0}|\theta-2\a_{i_0})=2-4(\theta|\a_{i_0})+4=2$.
\end{proof}
\begin{remark}\label{DES}
 A brief inspection of the Dynkin diagrams shows that  $\g^\natural$ does not have exactly two  irreducible  ideals  precisely when $\g$ belongs to the Deligne's series $A_2 \subset G_2 \subset D_4 \subset F_4 \subset E_6 \subset E_7 \subset E_8$ 
or it is of type $C_n$. The first case occurs exactly when there is  a long simple root not orthogonal to the highest root. 
\end{remark}
\begin{remark}
It is worthwhile to recall that $\dim \g=r(h+1)$, where $r$ is the rank of $\g$ and $h$ is the Coxeter number of $\g.$
\end{remark}

 \begin{proposition}\label{propfinale} Let $\{x_i\}$ be an orthonormal basis of $\g^\natural$. Then
 \begin{align*}
&[{G^v}_\l G^w]=\\&C\left(\tfrac{1}{4p(k)}\langle u,v\rangle\sum_i :J^{x_i}J^{x_i}:-\tfrac{1}{4p(k)}\sum_{i, j}(\langle[v,x_i],[w,x_j]\rangle+\langle[v,x_j],[w,x_i]\rangle):J^{x_i}J^{x_j}:\right)+\\
&C\left(-\tfrac{k+h^\vee}{2p(k)}\langle u,v\rangle L+\sum_r\tfrac{1}{k+\tfrac{h^\vee-\bar h_r^\vee}{2}}(\tfrac{1}{2}TJ^{[[x_\theta,v],w]_r^\natural}+\l J^{[[x_\theta,v],w]_r^\natural})+\tfrac{\l^2}{2}\langle u,v\rangle\right).
\end{align*}
\end{proposition}
\begin{proof}  Substitute the values \eqref{1}, \eqref{2}, \eqref{3} in the explicit expressions for $a(k),b(k),$ $c(k),d(k)$.
	\end{proof}

\begin{rem}
 Choosing $C=4p(k)$ one obtains a refinement of  the formula (1.1) of \cite{AKMPP} which is in turn an improvement of the original formula of Kac and Wakimoto \cite{KW}.
 
Indeed, recall that $(x_\theta|[u_r,u^s])=\d_{r,s}$. As in \cite{KW}, we let $\langle\cdot,\cdot\rangle_{ne}$ be the invariant form on $\g_{1/2}$ defined by setting $\langle v,w\rangle_{ne}=(x_{-\theta}|[v,w])$.
Note that 
$$
\langle [x_\theta,u_r],[x_\theta,u^s]\rangle_{ne}=-\tfrac{1}{2}\d_{r,s}.
$$
In fact,
$$
\langle [x_\theta,u_r],[x_\theta,u^s]\rangle_{ne}=(x_{-\theta}|[[x_\theta,u_r],[x_\theta,u^s]])=\tfrac{1}{2}(u_r|[x_\theta,u^s])=-\tfrac{1}{2}(x_\theta|[u_r,u^s])
$$
It follows that $\{[x_\theta, u^r]\}$ gives a basis of $\g_{1/2}$ and that $\{2[x_\theta,u_s]\}$ is its dual basis.

If $u\in\g_{-1/2}$ then
$$
[u,[x_\theta,u_s]]^\natural=\sum_i([u,[x_\theta,u_s]]|x_i)x_i=-\sum_i(x_\theta|[[x_i,u_s],u])x_i,
$$
$$
[[x_\theta,u^r],v]^\natural=\sum_i([[x_\theta,u^r],v]|x_i)x_i=\sum_i(x_\theta|[[x_i,u^r],v])x_i,
$$
so
\begin{align*}
2\sum_r:J^{[u,[x_\theta,u_r]]^\natural} J^{[[x_\theta,u^r],v]^\natural}:&=-2\sum_{i,j,r}(x_\theta|[[x_i,u_r],u])(x_\theta|[[x_j,u^r],v]):J^{x_i}J^{x_j}:\\
&=-2\sum_{i,j,r}(x_\theta|[u_r,[u,x_i]])(x_\theta|[u^r,[v,x_j]]):J^{x_i}J^{x_j}:\\
&=-2\sum_{i,j}(x_\theta|[[u,x_i],[v,x_j]]):J^{x_i}J^{x_j}:.
\end{align*}
\end{rem}
\begin{theorem}\label{finale} If $\g^\natural$  has one or three components, then 
$$\d=-\frac{1}{2}\left(\frac{h^\vee-\bar h_1^\vee}{2}\right)\left(\frac{\bar h^\vee_1}{2}+1\right),$$
while if $\g^\natural$ has two components, then
$$\d=-\frac{1}{2}\left(\frac{h^\vee-\bar h_1^\vee}{2}\right)\left(\frac{h^\vee-\bar h_2^\vee}{2}\right)$$
and $\bar h_1^\vee+\bar h_2^\vee=h^\vee-2$.
In particular, in both cases,
\begin{equation}\label{ddddd}p(k)=k^2 +\left(\frac{h^\vee}{2}+1\right)k-2\d.\end{equation}
\end{theorem}
\begin{proof}
Just substitute \eqref{Pan}, \eqref{1}, \eqref{2}, and \eqref{3} in formula \eqref{finaledelta}.
\end{proof}
\begin{rem}\label{CCPP} Let $V$ be a finite dimensional $Y(\g)$-module.  Using the (Hopf algebra) automorphism  $\tau_u,\,u\in\mathbb C$ defined by 
$\tau_u(x)=x,\,\tau_u(J(x))=J(x)+ux,\,x\in\g$, the  representation $V$ can be pulled back to give a one-parameter family of representations
$V(u)$.

Recall now that the R-matrix associated to $Y(\g)$-modules $V,W$ is of the form 
$R_{V,W}(u)=I_{V,W}(u)\sigma$,
where $\sigma$ is the switch automorphism and $I_{V,W}(u): V\otimes W(u)\to W(u) \otimes V$
is the unique intertwining operator which preserves the tensor product of the highest weight vectors in $V,W$. 
Since $I_{V,W}(u)$ is a $\g$-module map, it must preserve the isotypic components in $V\otimes W$. Denote by $(V\otimes W(u))_\g$, $(W(u)\otimes V)_\g$ the isotypic components corresponding to the adjoint representation and by $(V\otimes W(u))_0$, $(W(u)\otimes V)_0$ the isotypic components corresponding to the trivial representation. Set also 
$$P_\g=I_{V,W}(u)_{|(V\otimes W(u))_\g},\quad
P_0=I_{V,W}(u)_{|(V\otimes W(u))_0}.$$
Assume $\g$ is not $sl(n)$.
In the special case when $V=W=\mathcal V=\g\oplus\C$,  the adjoint representation occurs in  $\mathcal V\otimes\mathcal V$ with multiplicity three and the trivial representation with multiplicity two. 

As in Section 5.4 of \cite{ChPr1}, choose the following bases for the $\g$-highest weight spaces of $\mathcal V\otimes\mathcal V$ of weight $\theta$ and $0$:
$$\{x_\theta\otimes1+1\otimes x_\theta,\delta[x_\theta\otimes 1,C_\g],x_\theta\otimes1-1\otimes x_\theta\},\quad\quad\{1\otimes1,1\otimes1-\delta C_\g\}.
$$

One  the main results of \cite{ChPr1}  is the explicit computation of the matrices of $P_\g$ and $P_0$ in the bases given above. The final outcome, as far as we are concerned, is that the entries of these matrices are rational functions of $u$ whose denominator is either $g(u)=u^2-(\tfrac{h^\vee}{2}+1)u-2\delta=p(-u)$ (using \eqref{ddddd}), or $u-1$, or $(h^\vee/2-u)g(u)$.




\end{rem}

\vskip30pt
\setlength{\itemsep}{2.5pt} 
 
\vskip10pt
 \footnotesize{

\noindent{\bf V.K.}: Department of Mathematics, MIT, 77
Mass. Ave, Cambridge, MA 02139;\newline
{\tt kac@math.mit.edu}
\vskip5pt
\noindent{\bf P.MF.}: Politecnico di Milano, Polo regionale di Como,
Via Valleggio 11, 22100, Como, Italy;\newline {\tt pierluigi.moseneder@polimi.it}
\vskip5pt
\noindent{\bf P.P.}: Dipartimento di Matematica, Sapienza Universit\`a di Roma, P.le A. Moro 2,
00185, Roma, Italy;\newline {\tt papi@mat.uniroma1.it}
}


\begin{thebibliography}{KTWW}
\bibitem[AKMPP]{AKMPP} D.~Adamovi\'c, V.~G. Kac, P.~M\"oseneder Frajria, P.~Papi, O.~Per\v{s}e, {\em Conformal embeddings of affine vertex algebras in minimal
$W$-algebras I: Structural results}, J. Algebra, \textbf{500} (2018), 117--152.
\bibitem[CMP]{CMP} P.~Cellini, P.~M\"oseneder Frajria, P.~Papi, {\em
Abelian subalgebras in $\mathbb Z_2$-graded Lie algebras and affine Weyl groups
}, IMRN \textbf{43}, (2004), 2281--2304
\bibitem[CMPP]{CMPP} P.~Cellini, P.~M\"oseneder Frajria, P.~Papi, M.~Pasquali, {\em
On the structure of Borel stable abelian subalgebras in infinitesimal symmetric spaces}, Selecta Mathematica \textbf{19}, Issue 2 (2013), 399-437
\bibitem[CP]{ChPr1}
	V.~Chari, A.~Pressley, 
	{\em  Fundamental representations of Yangians and singularities of R-matrices},
	J. Reine Angew. Math. {\bf 417} (1991), 87--128. 	
\bibitem[DK]{DK} A.~De Sole, V.~G. Kac, {\em Finite vs affine $W$-algebras},
Japan. J. Math. \textbf{1} (2006), 137--261.	
\bibitem[Dr]{Dr1}
	V. Drinfeld, 
	{\em Hopf algebras and the quantum Yang-Baxter equation}, 
	Soviet Math. Dokl. \textbf{32} (1985), 254--258.	
\bibitem[GNW]{GNW} N.~Guay, H.~Nakajima, C.~Wendlandt, 
{\em
Coproduct for Yangians of affine Kac-Moody algebras.}
Adv. Math. \textbf{338} (2018), 865--911.

\bibitem[GRW]{GRW} N.~Guay, V.~Regelskis, C.~Wendlandt, {\em Equivalences between three presentations of orthogonal and symplectic Yangians.} Lett. Math. Phys. \textbf{109} (2019), no. 2, 327--379.
	 \bibitem[KRW]{KRW} V.~G. Kac, S.-S. Roan, M. Wakimoto, Quantum reduction for affine superalgebras, Comm.
Math. Phys. \textbf{241} (2003) 307--342.
 \bibitem[KW]{KW} V.~G. Kac, M.~Wakimoto \emph{
Quantum reduction and representation theory of superconformal algebras}, Adv. Math., \textbf{185} (2004), pp. 400--458.

\bibitem[Ko]{Ko} B.~Kostant, {\em Eigenvalues of a Laplacian and commutative Lie subalgebras},
Topology,
{\bf 3} (1965), 147--159.
 

\bibitem[P]{P} D.~I.~Panyushev,\emph{Isotropy representations, eigenvalues of a Casimir element, and commutative Lie subalgebras}, J. London Math. Soc., \textbf{64} no. 1 (2001), 61--80.
 \bibitem[R]{R} M.~Reeder, {\em Exterior powers of the adjoint representation.}
Canad. J. Math. \textbf{49} (1997), no. 1, 133--159. 	 


\end{thebibliography}
\end{document}